\documentclass[10pt]{article}
\newcommand\blfootnote[1]{%
	\begingroup
	\renewcommand\thefootnote{}\footnote{#1}%
	\addtocounter{footnote}{-1}%
	\endgroup
}

\usepackage{amsmath, amssymb, cases, color,amsthm}
\usepackage{hyperref}

\usepackage{bbm}
\usepackage{stmaryrd}
\usepackage{amsmath, amssymb,amscd, mathtools}
\usepackage{amsfonts}
\usepackage{graphicx}

 \usepackage{upref}
\hypersetup{linkcolor=blue, colorlinks=true,citecolor = red}
\newtheorem{thm}{Theorem}[section]
\newtheorem{prop}[thm]{Proposition}
\newtheorem{lem}[thm]{Lemma}

\newtheorem{rmk}[thm]{Remark}

\theoremstyle{definition}
\newtheorem{definition}[thm]{Definition}

\theoremstyle{remark}
\numberwithin{equation}{section}

\newcommand{\BE}{\begin{equation}}
\newcommand{\EEQ}{\end{equation}}
\newcommand{\rfb}[1]{\mbox{\rm
		(\ref{#1})}\ifx\undefined\stillediting\else:\fbox{$#1$}\fi}

\newfont{\roma}{cmr10 scaled 1200}

\newcommand{\pline}  {{\mathbb P}}
\newcommand{\rline}  {{\mathbb R}}

\newcommand{\tline}  {{\mathbb T}}

\newcommand{\dd}  {{\rm d}\hbox{\hskip 0.5pt}}
\renewcommand{\leq} {\leqslant}
\renewcommand{\geq} {\geqslant}
\newcommand{\mm}    {{\hbox{\hskip 0.5pt}}}
\newcommand{\m}     {{\hbox{\hskip 1pt}}}

\newcommand{\bluff} {{\hbox{\raise 15pt \hbox{\mm}}}}
\newcommand{\sbluff}{{\hbox{\raise 10pt \hbox{\mm}}}}
\newcommand{\Om}    {{\Omega}}
\renewcommand{\div} {{\rm div\,}}
\newcommand{\eps}    {{\varepsilon}}
\newcommand{\ue}     {u_{\varepsilon}}

\newcommand{\lep}    {l_\varepsilon}
\newcommand{\ome}    {\omega_\varepsilon}
\newcommand{\dx}     {\mathrm{d}x} 
\newcommand{\ds}     {\mathrm{d}s} 
\newcommand{\dy}     {\mathrm{d}y} 

\def\toep{\stackrel{\eps\to0}{\longrightarrow}}
\def\toepk{\stackrel{\eps_k\to0}{\longrightarrow}}

\newcommand{\FORALL} {{\hbox{$\hskip 11mm \forall \;$}}}

\newcommand{\prt}      {{\partial}}

\def\loc{_{\text{loc}}}
\DeclarePairedDelimiter{\norm}{\lVert}{\rVert}

\newcommand{\Ascr} {\mathcal{A}}

\newcommand{\Dscr} {\mathcal{D}}
\newcommand{\Fscr}{\mathcal{F}}

\newcommand{\Kscr} {\mathcal{K}}

\newcommand{\Oscr} {\mathcal{O}}

\newcommand{\Sscr} {\mathcal{S}}

\title{The vanishing limit of a rigid body in three-dimensional viscous incompressible fluid \blfootnote{The authors would like to thank Sylvain Ervedoza (Universit\'e de Bordeaux) for helpful discussions and also Tak\'eo Takahashi (Inria Nancy) for suggestions. P. Su is partially supported by ERC-CZ grant LL2105 and the University Centre UNCE/SCI/023 of Charles University.}}
\usepackage{authblk}
\author{Jiao He$^1$}

\author{Pei Su$^2$\thanks{corresponding author}}

\affil{\footnotesize $^1$ Laboratoire de Mathématiques d’Orsay, Université Paris-Saclay \\
307 Rue Michel Magat, Orsay 91400, France \\
ORCID : 0000-0001-8951-4566
\\ jiao.he@universite-paris-saclay.fr}

\affil{\footnotesize $^2$ Department of Mathematical Analysis, Faculty of Mathematics and Physics, Charles University \\
Sokolovská 83, 186 75 Praha 8, Czech Republic\\
ORCID : 0000-0002-5817-4260\\
supeiamss@gmail.com}

\begin{document}
 \date{\today}



\maketitle
\begin{scriptsize}
\abstract{	
We consider the evolution of a small rigid body in an incompressible viscous fluid filling the whole space $\rline^3$. When the small rigid body  shrinks to a \textit{``massless" point} in the sense that its density is constant, we prove that the solution of the fluid-rigid body system converges to a solution of the Navier-Stokes equations in the full space. Based on some $L^p-L^q$ estimates of the fluid-structure semigroup and a fixed point argument, we obtain a uniform estimate of velocity of the rigid body. This allows us to construct admissible test functions  which plays a key role in the procedure of passing to the limit. }
\end{scriptsize}

\




\section{Introduction}\label{sec_intro}
In this work, we investigate the interaction between a viscous 
incompressible fluid and a small rigid body, which is completely 
immersed in the fluid. The fluid we consider here is governed 
by the three dimensional Navier-Stokes equations. The motion
of the rigid body obeys the conservation of linear and angular
momentum, in particular, without taking into account the influence 
of the gravity. Here we are interested in describing the dynamics of 
the fluid-body system as the rigid body shrinks to a point. 
Now let us precisely state the problem. 

\subsection{Statement of the problem}
We assume that the rigid body is a ball of radius $\eps>0$ and 
centered at $h_\eps(t)\in\rline^3$, occupying the area
$$ \Sscr_\eps(t)=\overline{B(h_\eps(t), \eps)}$$
in $\rline^3$. The fluid domain, denoted by $\Fscr_\eps(t)$, is 
the exterior part of $\Sscr_\eps(t)$, i.e. 
$$\Fscr_\eps(t)=\rline^3\setminus \Sscr_\eps(t).$$
We denote by $n(t,x)$ the unit normal vector field of 
$\prt\Fscr_\eps(t)$, which directs toward the interior of 
$\Sscr_\eps(t)$ and is independent of the size of the body $\eps$. 
The fluid is supposed to be homogeneous with 
density $\rho_\eps^f=1$ and constant viscosity $\nu>0$.
The velocity and the pressure in the fluid 
are denoted by $u_\eps(t,x)$ and $p_\eps(t,x)$, respectively.
With the above notation, the system describing
the motion of the rigid ball in the fluid, for all $t>0$, reads
\begin{itemize}
\item Fluid equations:
\begin{flalign}{\label{ns-ob1}}
\begin{split}
\partial_t u_{\eps} +(u_{\eps} \cdot \nabla )
u_{\eps}- \nu \Delta u_{\eps} + \nabla p_{\eps} =0 \FORALL  y 
\in \Fscr_{\eps}(t), \\
\div u_{\eps} = 0   \FORALL  y \in \Fscr_{\eps}(t),
\end{split}
\end{flalign}
\item Rigid body equations:
\begin{align}
m_{\eps} \ddot h_{\eps}(t) = - \int_{\partial \Sscr_{\eps}(t)} 
\sigma(u_\eps,p_\eps) n\m \dd s, \;\;\;{\label{solid1}} \\
J_{\eps}\dot \omega_{\eps}(t) = - \int_{\partial \Sscr_{\eps}(t)} 
(y-h_{\eps}) \times (\sigma (u_\eps,p_\eps) n)\m \dd s, 
\;\;\; {\label{solid2}}
\end{align}
\item Boundary conditions:
\begin{flalign}\label{boundarycon1}
\begin{split}
u_{\eps}(t,\m y)= \dot h_{\eps}(t)+ \omega_{\eps}(t) \times 
\left(y-h_{\eps}(t)\right)  \FORALL y \in \partial \Sscr_{\eps}(t).
\end{split}
\end{flalign}
\item Initial conditions:
\begin{equation}{\label{initialcon1}}
\begin{aligned}
u_{\eps}(0,\m y) = u_{\eps}^0(y) \FORALL y\in\Fscr_\eps(0),\\
\quad
h_{\eps} (0)=0, \quad 
\dot h_{\eps} (0) = \lep^0, \quad
\omega_\eps(0) = \ome^0.
\end{aligned}
\end{equation}
\end{itemize}
Here we assume that the initial position of the center 
of mass for the rigid body is at the origin. In the above equations, $\omega_\eps(t)$ represents the angular 
velocity of the body. Let us denote by $\rho^{s}_\eps$ the density of the rigid body, which is assumed to be constant, i.e., $\rho^{s}_\eps = \rho$. 
The constant $m_\eps$ and the matrix 
$J_\eps$ stand for the mass and the inertia tensor of the 
rigid ball and can be expressed as below, respectively: 
\begin{equation}\label{formulamass}
\begin{aligned}
m_\eps&=\int_{\Sscr_\eps (t)}\rho\m \dd y,\\
 (J_\eps)_{i,j}
=\int_{\Sscr_\eps(t)}\rho(I_3 |y-h_\eps(t)|^2-&(y_i-h_\eps(t))(y_j-h_\eps(t)))\m\dd y
\quad i,j= 1,2,3.
\end{aligned}
\end{equation}
The Cauchy stress tensor of the fluid 
$\sigma (u_\eps, p_\eps)$ is given by 
$$
\sigma (u_\eps, p_\eps) = 2 \nu D(u_\eps) - p_\eps I_{3},
$$ 
where $I_{3}$ is identity matrix of order $3$ and $D(u_\eps)$ 
is the deformation tensor
\begin{equation}\label{symmetric}
D(u_\eps):=\frac{1}{2}\left(\nabla u_\eps + (\nabla u_\eps)^
\intercal\right).
\end{equation}

We introduce here some notation which is used throughout this 
paper. For any smooth open set $\Oscr\subset\rline^3$, the notation 
$L^p(\Oscr)$ and $W^{s,p}(\Oscr)$, for every $s\in\rline$ and 
$1\leq p<\infty$, represents the standard Lebesgue and 
Sobolev-Slobodeckij spaces,  respectively. For $p=2$, as usual we use the notation $H^s(\Oscr)$. Moreover, we define the divergence 
free spaces as follows:
\begin{equation}\label{spacefree}
\begin{aligned}
L_\sigma^p(\Oscr)=\left\{ f\in L^p(\Oscr)\m\m \left|\m\m \div f=0 
\,\,\, \text{in}\,\,\, \Oscr\right.\right\} ,\\
H_\sigma^s(\Oscr) =\left\{f\in H^{s}(\Oscr)\m\m\left|\m\m 
\div f=0 \,\,\, \text{in}\,\,\, \Oscr \right.\right\}.
\end{aligned}
\end{equation}
For the sake of simplicity, we denote by $\lVert~\cdot~
\rVert_{p,\Oscr}$ and $|\cdot|$ the norm in $L^p(\Oscr)$ and in 
$\rline^3$, respectively. For a matrix $M$, we denote by 
$M^\intercal$ the transpose of $M$. Finally, if a function $f$ 
only depends on time $t$, we denote by $\dot f$ its derivative 
with respect to $t$; and if $f$ depends on both time and space, 
we use $\prt_t f$ for its partial time-derivative.

We assume that the initial data in \rfb{initialcon1} satisfy 
\begin{equation}\label{inicond}
\begin{aligned}
u_\eps^0\in L^2(\Fscr_\eps(0)),\quad \div u_\eps^0=0\quad \text{in}
\quad \Fscr_\eps(0),\\
u_\eps^0\cdot n=\left( l_\eps^0+\omega_\eps^0\times y \right)
\cdot n \quad \text{on}\quad \prt\Sscr_\eps(0).
\end{aligned}
\end{equation}

\begin{rmk}[Energy estimates]\label{remark_energy}
{\rm 
Taking the inner product of \eqref{ns-ob1} with $u_\eps$, integrating the result by parts and using the equations \eqref{solid1} and \eqref{solid2}, we get the following energy estimate (see, for instance, Gunzburger et al. \cite{gunzburger2000global}):
\begin{multline*}\label{energinequality}
\|u_\eps (t)\|_{L^2(\mathcal{F}_{\eps}(t))}^2 + m_\eps |\dot{h}_\eps (t)|^2 + \left(J_\eps\m \omega_{\eps}(t) \right) \cdot \omega_{\eps}(t) + 4\m\nu \int_0^t \|D(u_\eps)\|_{L^2(\mathcal{F}_{\eps}(t))}^2 \\
\leq \|u_\eps^0\|_{L^2(\mathcal{F}_{\eps}(0))}^2 + m_\eps  |\m\lep^0|^2 + \left(J_\eps  \omega_{\eps}^0 \right) \cdot \omega_{\eps}^0.
\end{multline*}
Note that the tensor of inertia $J_\eps$ is positive. 
We observe that if the initial data $\ue^0$, $l_\eps^0$ and $\omega_{\eps}^0$ are bounded, we then have the boundedness of $u_\eps (t)$ and $m_\eps |\dot{h}_\eps (t)|^2$. Nevertheless, the energy estimate above is not enough to obtain a uniform estimate of the velocity of the rigid body i.e. $\dot{h}_\eps (t)$ when the mass of the rigid body tends to zero.}
\end{rmk}

\subsection{Weak solutions and main results}
In order to introduce the {\em Leray-Hopf weak solution} and state the existence of solutions of the system \eqref{ns-ob1}-\eqref{initialcon1}, let us first extend the velocity field, still denoted by $\ue$, as follows:
\begin{equation*}
\ue(t,y)=\mathbbm{1}_{\Fscr_\eps(t)}u_\eps(t,y)+\mathbbm{1}_{\Sscr_\eps(t)}\left(\dot h_{\eps}(t)+\omega_{\eps}(t)\times\left(y-h_{\eps}(t) \right)\right) \FORALL y\in\rline^3,
\end{equation*}
where $\mathbbm{1}_A$ is the indicator function of the set 
$A$.
Clearly, with conditions \eqref{inicond}, we know that 
\begin{equation*}
\ue^0 \in L^{2}(\rline^{3}), \quad \div \ue^0 = 0\;\; \text{in} \;\; \rline^{3}.\\
\end{equation*}
We also define the global density of the fluid-rigid body system in $\rline^3$ as follows:
\begin{equation}\label{densityt}
\rho_\eps(t,y)=\mathbbm{1}_{\Fscr_\eps(t)}(y) +\rho 
\mathbbm{1}_{\Sscr_\eps(t)}(y)\FORALL t>0, \,\,\, y\in\rline^3.
\end{equation}

Now we introduce the definition of a weak solution of the system \eqref{ns-ob1}-\eqref{initialcon1}. 

\begin{definition}[Leray-Hopf weak solution]\label{defu}
We call the triplet $(u_\eps, h_\eps, \ome)$ a global Leray-Hopf weak 
solution of the system \eqref{ns-ob1}-\eqref{initialcon1}, if for 
every $T>0$ we have
$$u_\eps\in L^\infty(0, T; L^2(\rline^3))\cap L^2(0, T; 
H^1(\rline^3)),$$
$$u_\eps(t,y)=\dot h_\eps(t)+\ome(t) \times (y-h_\eps(t)) 
\FORALL t\in[0,T),\,\,\, y\in \Sscr_\eps(t),$$
and $\ue$ verifies the equation in the following sense:
\begin{multline*}
-\int_{0}^{T} \int_{\rline^3} \rho_{\eps} \ue \cdot \left(\prt_s
\m\varphi_{\eps} + (\ue \cdot \nabla) \m\varphi_{\eps} 
\right) \, \dy \m\ds
+ 2 \nu \int_{0}^{T} \int_{\rline^{3}} D (\ue) : D (\varphi_{\eps})  
\,  \dy \m\ds 
= \int_{\rline^{3}} \rho_{\eps}\m \ue^0(y) \cdot \varphi_{\eps}(0, y) 
\,  \dy. 
\end{multline*}
for every $\varphi_\eps\in C_c^1([0,T); H_\sigma^1(\rline^3))$ 
such that $D(\varphi_\eps)=0$ in $\Sscr_\eps(t)$.
\end{definition}
There is a large literature about the existence of weak solutions of fluid-rigid body system in the past few decades. The first result about such problem has been established by Serre in \cite{serre1987chute}, where the author proved the global in time existence of Leray weak solution. 
Here we do not mention more about the existence issue, but state the following result for self-contained reason. For more details, please refer to \cite{conca2000existence, desjardins2000weak, gunzburger2000global, serre1987chute}.
\begin{thm}[Existence result]\label{Existence}
Let $\ue^0\in L^2(\rline^3)$ be divergence free such that $D(\ue^0)=0$ in $\mathcal{S}_\eps(0)$. 
Then there exists at least one global weak solution $(\ue, h_{\eps}, \omega_{\eps})$ for the initial-boundary value problem \eqref{ns-ob1}-\eqref{initialcon1} in the sense of Definition \ref{defu}. Moreover, $\ue$ satisfies the following energy estimate:
\begin{align*}{\label{energyineq}}
\int_{\rline^{3}} \rho_{\eps} |\ue(t)|^{2}  + 4 \nu \int_{0}^{t} \int_{\rline^{3}} |D (\ue)|^{2} \leq \int_{\rline^{3}} \rho_{\eps}(0) |\ue^0|^{2} \FORALL t>0.
\end{align*}
\end{thm}

Our purpose in this article is to investigate the asymptotic limit of the system \eqref{ns-ob1}-\eqref{initialcon1} when the size of the rigid body goes to $0$. Before stating our main result, let us first review some related references. The vanishing limit of small obstacle problem has been studied by a number of work, in both incompressible and compressible cases, with various assumptions on the fluid and, in particular, on the parameters of the body. 

When the small rigid body is fixed inside the fluid, for instance, Iftimie et al. \cite{iftimie2006two} proved the limit of $\ue$ satisfies the 2D Navier-Stokes equations, with assumptions on the initial vorticity and the circulation of the velocity around the body. Later on, Iftimie and Kelliher \cite{iftimie2009remarks} considered the same problem in 3D. Moreover, the influence of the precise shape of the small obstacle in the fluid is of particular interest. Lacave studied the case of a thin obstacle tending to a curve in \cite{lacave2009two} for 2D and in \cite{lacave2015three} for 3D. More recently,  Chipot et al. considered in \cite{chipot2020limits}  the vanishing limit problem in a 2D periodic flow where the domain of the fluid is punctured by a star-shaped open connected set.

In the case that the obstacle can move under the influence of the fluid, the asymptotic behaviour of the fluid-rigid body system becomes more complicated. Lacave and Takahashi in  \cite{lacave2017small} considered a small moving disk into a viscous incompressible fluid in 2D with 
the assumption that the mass of the obstacle tends to zero along with its size and some small initial data. When the density of the body goes to infinity as its size goes to zero, the convergence has been proved by He and Iftimie \cite{he2019small} for 2D and \cite{he2021small} for 3D only using energy estimates and without small initial data. More recently,  
Feireisl et al. \cite{feireisl2022motion3} extended the result of \cite{lacave2017small} by removing the restriction on the initial data and the shape of the rigid body for both 2D and 3D. In particular, their assumptions allow the density of the body to be asymptotically small. Moreover, Bravin and Ne{\v{c}}asov{\'a} 
\cite{bravin2022velocity} obtained some estimates of the velocity of the rigid body under the assumptions $m_\eps \eps^{-1/2} \to +\infty$ for 3D and $m_\eps \geq C >0$ for 2D, which seems helpful for vanishing rigid body problem.
In addition, there are some recent work that dealt with the similar problem in viscous compressible fluid, please refer to \cite{bravin2020vanishing, feireisl2022motion1, feireisl2022motion2} for more details.


Intuitively, the vanishing process of the body should be induced directly by the decreasing of the volume of the body, rather than the variation of the density. With constant density, the technique employed in \cite{lacave2017small} is the so-called $L^p-L^q$ estimate of the {\em fluid-structure semigroup} established in Ervedoza et al. \cite{ervedoza2014long}, which gives a uniform estimate of the obstacle velocity.  
Here we are interested in studying the corresponding problem with the similar setting in 3D, based on the recent long time behaviour work by Ervedoza et al. \cite{ervedoza2022large}. We finally remark that the long-time behavior of solutions to a fluid-rigid disk system in 2D case has been also studied in Ferriere and Hillairet \cite{ferriere2022unbounded} with somewhat different setting.

Still using the notation introduced at the beginning of this section, our main theorem can be stated as follows:

\begin{thm}\label{main-thm}
There exists a constant $\gamma_0>0$ such that if 
\begin{itemize}
 \item $\ue^0(x)\in L^2(\mathcal{F}_\eps (0))\cap L^3(\Fscr_\eps (0))$ is divergence free and satisfies $u_\eps^0\cdot n=\left( l_\eps^0+\omega_\eps^0\times y \right)
\cdot n$ on $\partial \mathcal{S}_\eps(0)$;
\item $\ue^0(x)$ converges weakly in $L^2(\rline^3)$ to some $u_0(x)$;
\item we have the following smallness of the initial data
\begin{equation*}\label{smalllt}
\norm{\ue^0(x)}_{L^2(\mathcal{F}_\eps(0)) \cap L^3(\Fscr_\eps(0))},\m |l_\eps^0|,\m |\omega^0_\eps|\leq \gamma_0,
\end{equation*}
\end{itemize}
then, for any $T >0$, we have
$$u_\eps \rightharpoonup u \quad \text{weak-$\ast$ in} \quad L^\infty(0, T;L^2(\rline^3))\cap L^2(0, T;H^1(\rline^3)),$$
where $u$ is a weak solution of the Navier-Stokes equations in $\rline^3$ with initial data $u_0$, i.e., $u$ satisfies
\begin{multline}\label{energy-eq-u}
-\int_{0}^{T} \int_{\rline^3} u \cdot \left(\prt_s
\m\varphi  + (u \cdot \nabla) \m\varphi
\right) \, \dy \m\ds
+ 2 \nu \int_{0}^{T} \int_{\rline^{3}} D (u) : D (\varphi)  
\,  \dy \m\ds 
= \int_{\rline^{3}}\m u^0(y) \cdot \varphi(0, y) 
\,  \dy, 
\end{multline}
for any $\varphi \in C_c^1 ([0, T); H^1_\sigma(\rline^3))$.
\end{thm}
The main novelty brought by the present paper is twofold. Firstly, there is no specific assumption for the density of the rigid body $\rho$, which is constant and independent of its size $\eps$. That is to say, we deal with the case where the rigid body shrinks to a \textit{massless} point:
$$m_\eps = \eps^3 m_1, \qquad J_\eps = \eps^5 J_1.  $$
Secondly, 
as already mentioned in Remark \ref{remark_energy}, energy estimate is not enough to get such estimate when the mass changes along with $\eps$. We thus showed that the $L^p-L^q$ estimate, associated with the \textit{fluid-structure semigroup}, in 3D is $\eps-$invariant, i.e. independent of the size of body, allowing us to estimate the velocity of the rigid body in a uniform way. This therefore implies the uniform convergence of $h_\eps$, i.e., up to a subsequence, as $\eps\to 0$,
$$h_{\eps}\longrightarrow h \quad \text{uniformly in}\quad [0,T],$$
where $h\in W^{1, q}(0, T)$ for $1<q<\frac{4}{3}$.
Please refer to subsection \ref{subsection-unif3} for more details.

\begin{rmk}
{\rm By constructing an appropriate approximation sequence of the test function, we actually obtain a local strong convergence for the solution of the fluid-body system, up to a subsequence, i.e. for every $T>0$,
$$ u_{\eps_k} \longrightarrow u\quad\text{strongly in }  L^{2}(0,T;L^2\loc(\rline^3)).$$
The proof is presented in Section \ref{section-strong}.
}
\end{rmk}

\begin{rmk}
{\rm Although our main result is stated for constant density, presumably our proof still works for the case where the density of the rigid body is larger than some positive $\rho_0$ independent of $\eps$. Moreover, when the rigid body shrinks to \textit{massive} point:
$$m_\eps = m_1, \qquad J_\eps = \eps^2 J_1,$$
a similar result can be proven without using the $L^p-L^q$ analysis. 
The above results should hold as well in the case of several bodies, as long as there is no collision.}
\end{rmk}


\subsection{Organization of the paper}
In Section \ref{section-uniform}, we first introduce change of variables to formulate the fluid-body system in a fixed space domain. In particular, we give the uniform estimate for the solid velocity which is based on the properties of the fluid-structure semi-group. Then we constructed in Section \ref{section-testfunc} an approximation sequence of the test function for the fluid-body system and this will play a key role in the justification of the limit. To deal with the nonlinear term, we show a strong convergence of subsequence for the solution of fluid-body system in Section \ref{section-strong}.
Section \ref{section-limit} is devoted to passing to the limit from the fluid-body system to the pure fluid system.
\section{Uniform estimates on the Solid velocity}\label{section-uniform}
In this section, we shall reformulate the fluid-body system in an abstract 
differential equation by using the so-called fluid-structure operator and 
the fluid-structure semi-group introduced in \cite{ervedoza2022large}. 
The main point is to obtain the uniform estimate of the velocity of the 
body, which plays an important role in the construction of the cut-off 
argument in Section \ref{section-testfunc}.

\subsection{Change of variables}\label{subsection-unif1}

We first notice that, via a change of variables, the fluid-solid 
system \eqref{ns-ob1}-\eqref{initialcon1} can be written in a fixed spatial 
domain $\rline^3 = \Fscr_{\eps}^0 \cup \Sscr_{\eps}^0 $ with 
$\Fscr^0_\eps=\Fscr_\eps(0)$ and $\Sscr_\eps^0=\Sscr_\eps(0)$. 
Setting $x \mapsto y (t, x) := x + h_\eps(t)$ and 
$$
v_{\eps} (t, \m x) = u_{\eps} (t, \m x + h_\eps(t)), \quad
\pi_\eps (t, \m x) = p_\eps (t, \m x + h_\eps(t)),
$$
$$v_{\eps}^0 (x) = u_{\eps}^0 (x),  \quad
\dot h_{\eps} (t)= \lep (t),
$$
the system \eqref{ns-ob1}-\eqref{initialcon1}, for every $t>0$, reads, 
\begin{subequations}\label{governing}
\begin{alignat}{15}
\partial_t v_{\eps} + \left[(v_{\eps} - \lep)\cdot 
\nabla\right]v_{\eps}- \nu \Delta v_{\eps} + \nabla \pi_{\eps} =0  
& \quad \FORALL x \in \Fscr_{\eps}^0,  \label{eq01} 
\\
\div v_{\eps} = 0   & \quad
\FORALL x \in \Fscr_{\eps}^0,  \label{eq02} 
\\
m_{\eps} \dot \lep (t) = - \int_{\partial \Sscr_{\eps}^0} 
\sigma(v_\eps, \pi_\eps)\m n\m \ds, & \quad \label{eq03} 
\\
J_{\eps}\dot \omega_{\eps}(t) = - \int_{\partial \Sscr_{\eps}^0} x 
\times (\sigma (v_\eps, \pi_\eps)\m n)\m \ds, & \quad  \label{eq04} 
\\
v_{\eps}(t,x)=\lep (t)+ \omega_{\eps}(t) \times x & \quad 
\FORALL x \in \partial \Sscr_{\eps}^0,  \label{eq05} 
\\
v_{\eps}(0, x) = v_{\eps}^0, \quad
\lep (0) = \lep^0, \quad
\omega_\eps(0) = \ome^0 & \quad \FORALL x \in \Fscr_{\eps}^0. 
\label{eq06} 
\end{alignat}
\end{subequations}
In the above fixed spacial framework, the global density introduced 
in \rfb{densityt} becomes
\begin{equation*}\label{newdensity}
\rho_\eps(x)=\mathbbm{1}_{\Fscr_\eps^0}(x) +\rho 
\mathbbm{1}_{\Sscr_\eps^0}(x) \FORALL x\in\rline^3.
\end{equation*}
Now we define the weak solution of the system \rfb{governing} as 
follows.

\begin{definition}\label{defv}
The triplet $(v_\eps, l_\eps, \omega_\eps)$ is a weak solution of 
the system \rfb{governing}, if, for every $T>0$, we have 
$$ v_\eps\in L^\infty\left(0, T; L^2(\rline^3)\right)\cap 
L^2\left(0, T; H^1(\rline^3)\right), $$
$$\quad v_\eps(t,x)=l_\eps(t)+\omega_\eps(t) \times x
\FORALL t\in[0,T), \quad x\in \Sscr_\eps^0,  $$
and $v_\eps$ verifies the equation in the following sense:
\begin{multline*}
-\int_{0}^{T} \int_{\rline^{3}} \rho_{\eps} v_\eps \cdot \left( 
\prt_t\m\varphi_{\eps}  + \left[\m(v_\eps - \lep)  \cdot 
\nabla\m\right] \varphi_{\eps} \right)\m
\dx \m \ds \\+ 2 \nu \int_{0}^{T} \int_{\rline^{3}} D (v_\eps) 
: D (\varphi_{\eps})  \,  \dx\m \ds 
= \int_{\rline^{3}} \rho_{\eps}\m v_\eps^0(x) \cdot 
\varphi_{\eps}(0, x)  \,  \dx,
\end{multline*}
for every $\varphi_\eps(t,x)\in C_c^1([0,T); H^1_\sigma(\rline^3))$ 
such that $D(\varphi_\eps)=0$ in $\Sscr_\eps^0.$
\end{definition}
It is not difficult to check that $u_\eps$ is a weak solution of 
\rfb{governing} in Definition \ref{defu} if and only if $v_\eps$ 
is a weak solution of \eqref{ns-ob1}-\eqref{initialcon1} in the 
sense of Definition \ref{defv}. 

Recalling the tensor $D$ introduced in \rfb{symmetric}, for every 
$1< q<\infty$ we define the function space
\begin{equation*}
X^q_\eps:=X^q_\eps(\rline^3)=\left\{ \Phi\in  L_\sigma^q
(\rline^3) \m\left|\m D(\Phi)=0 \,\,\,\text{in}\,\,\, 
\Sscr_\eps^0\right.\right\},
\end{equation*}
where the space $L_\sigma^q$ has been introduced in \rfb{spacefree}. 
Since every $\Phi\in X_\eps^q$ satisfies $D(\Phi)=0$ in 
$\Sscr_\eps^0$, there exists a unique couple $\begin{bmatrix}\m l 
&  \omega\m \end{bmatrix}^\intercal\in \rline^3\times\rline^3$ and 
$\varphi\in L_\sigma^q(\Fscr_\eps^0)$ such that 
\begin{equation}\label{extension}
\Phi(x)=\varphi(x)\mathbbm{1}_{\Fscr_\eps^0}(x)+(l+\omega\times x)
\mathbbm{1}_{\Sscr_\eps^0}(x).
\end{equation}
For the existence of $l$ and $\omega$, please refer to, for 
instance, Temam's book \cite[Lemma 1.1]{temam2018mathematical}. 
Therefore, the space $X_\eps^q$ can be specified as follows:
\begin{equation*}\label{Xeps}
 X_\eps^q=\begin{Bmatrix}\m\begin{bmatrix} \varphi \\ l \\ 
 \omega\end{bmatrix}\in L^q(\Fscr_\eps^0)\times 
 \rline^3\times\rline^3\m\left| \m\m\div\varphi=0 \,\,\,\,\,
 \text{in}\,\,\,\Fscr_\eps^0\right. \\
\,\,\,\text{and}\quad\varphi(x)\cdot n=(l+\omega\times x)\cdot 
n\quad\forall\m\m x\in\prt\Sscr_\eps^0 \end{Bmatrix},   
\end{equation*}
endowed with the norm
$$\lVert \Phi\rVert_{X_\eps^q}=\lVert \varphi\rVert_{q,\Fscr_\eps^0}
+|\m l\m|+|\m\omega\m|. $$

Note that $L^q(\rline^3)$ can be decomposed into 
$X_\eps^q$ with another two potential form spaces, we define the 
projection operator:
\begin{equation}\label{projection}
\mathbb{P}_\eps: L^q(\rline^3)\longrightarrow X_\eps^q.
\end{equation}
For more details about this decomposition, please refer to Wang 
and Xin \cite[Theorem 2.2]{wang2011analyticity} or Ervedoza 
et al. \cite[Proposition 3.1]{ervedoza2022large}.

To reformulate the equations \eqref{eq01}--\eqref{eq06}, for every 
$1< q<\infty$ we define the operator $\mathcal{A}^q_\eps : \Dscr(\mathcal{A}^q_\eps)\longrightarrow L^q (\rline^3) $  by
\begin{equation*}
	\mathcal{D}(\mathcal{A}^q_\eps):=\left\{ v\in W_0^{1,q}
	(\rline^3)\cap X_\eps^q\m\left|\m\m v|_{\Fscr_\eps^0}
 \in W^{2,q}(\Fscr_\eps^0)\right. \right\}, 
\end{equation*}
and
\begin{equation*}
	\mathcal{A}^q_\eps v :=
	\left\{
	\begin{array}{ll}
		-\nu \m\Delta v & \text{in}\ \Fscr_\eps^0,\\[0.2cm]
		\displaystyle \frac{2\m\nu}{m_\eps} \int_{\partial 
  \Sscr_\eps^0} D(v)n \, \ds + 
		\frac{2\m\nu}{J_\eps} \left(\int_{\partial \Sscr_\eps^0} 
  x\times\left(D(v)n\right) \, \ds\right)\times x & \text{in}\ \Sscr_\eps^0,
	\end{array}
	\right. 
\end{equation*}
for every $v\in\Dscr(\Ascr^q_\eps)$.

The {\em fluid-structure operator} $A^q_\eps 
: \Dscr(A^q_\eps)\longrightarrow X_\eps^q$ is defined by 
$$\Dscr(A^q_\eps)
=\Dscr(\Ascr^q_\eps),$$
and 
\begin{equation}\label{f-soper}
\begin{aligned}
A^q_\eps:= \mathbb{P}_\eps \mathcal{A}^q_\eps.
\end{aligned}
\end{equation}

For $v_\eps\in L_\sigma^q(\Fscr_\eps^0)$ with $D(v_\eps)=0$ in 
$\Sscr_\eps^0$, as we explained around \rfb{extension}, we denote 
by $V_\eps$ the extension of $v_\eps$, 
\begin{equation}\label{extendv}
V_\eps(t,x)=v_\eps(t, x)\mathbbm{1}_{\Fscr_\eps^0}(x)+(l_{v_\eps}(t)
+\omega_{v_\eps}(t)\times x)\mathbbm{1}_{\Sscr_\eps^0}(x).
\end{equation}
It is not difficult to see that we have the following relation:
$$ v_\eps(t,x)=V_\eps|_{\Fscr_\eps^0},\quad l_{v_\eps}(t)=
\frac{1}{m_\eps}\int_{\Sscr_\eps^0} V_\eps\m \dd x,\quad 
\omega_{v_\eps}(t)=-\frac{1}{J_\eps}\int_{\Sscr_\eps^0}V_\eps
\times x\m\dd x. $$
By using the fluid-structure operator $A_\eps^q$ defined above, 
together with the extension of $v_\eps$ in \rfb{extendv}, one 
can rewrite the system \eqref{eq01}--\eqref{eq06} in the 
following abstract form:
\begin{equation}\label{abstract}
\left\{
\begin{aligned}
&\m \dot V_{\eps}+ A_\eps^q V_{\eps}= \mathbb{P}_\eps \div 
F_\eps(v_\eps),\\
&\m V_\eps(0)=V^0_\eps,
\end{aligned}
\right.
\end{equation}
with
\begin{equation}\label{eq:Feps}
	F_\eps (v_\eps)= 
	\left\{
	\begin{aligned}
		&v_\eps\otimes (\ell_{v_\eps}- v_\eps) & \text{ 
  on $\Fscr_\eps^0$},\,\\
		&0  &\text{ on $\Sscr_\eps^0$}.
	\end{aligned}\right.
\end{equation}

According to \cite[Theorem 6.1]{ervedoza2022large}, the operator 
$A_\eps^q:\Dscr(A_\eps^q)\longrightarrow X_\eps^q$ in 
\rfb{f-soper} generates a bounded analytic semi-group 
$\tline_\eps^q=(\tline^q_{\eps,t})_{t\geq0}$ on $X_\eps^q$, 
which is called the {\em fluid-structure semi-group}.
Formally, we have the formula of the mild solution for \rfb{abstract}:
\begin{equation}\label{mild}
V_\eps(t)=\tline^q_\eps(t)V_\eps^0+\int_0^t\tline^q_\eps(t-s)
\mathbb{P}_\eps \div F_\eps(v_\eps(s))\dd s \FORALL t\geq0.
\end{equation}

\begin{rmk}
{\rm The projection operator $\mathbb{P}_\eps$ introduced in 
\rfb{projection} is actually a Leray type projection with 
additional condition $D(v_\eps)=0$ in $\Sscr_\eps^0$. Therefore, 
in the definition of $\Ascr_\eps^q$ above, the function $D(v)n$ 
can be replaced by $\sigma(v, \pi)n$ since the operator 
$\mathbb{P}_\eps$ vanishes for the vector fields coming from 
a potential, i.e. $\mathbb{P}_\eps (\nabla p_\eps) =0$.  }
\end{rmk}

\subsection{Decay estimates for the fluid-structure semi-group}
\label{subsection-unif2}
In this subsection, we quote the $L^p-L^q$ estimates of the 
fluid-structure semi-group from 
\cite[Theorem 7.1 and Lemma 8.1]{ervedoza2022large}. Since here 
we are concerned with the shrinking limit problem, we need to 
show that these estimates are independent of the size of the body 
$\eps$. This is the key tool to obtain the uniform estimates of 
the solid velocity. 

Recalling the formula of the mass $m_\eps$ and the inertia matrix $J_\eps$ in \rfb{formulamass}, after change of variables it becomes
$$m_\eps=\int_{\Sscr_\eps^0}\rho\m\dd x,\quad\quad (J_\eps)_{i, j}=\int_{\Sscr_\eps^0}\rho\m(I_3|x|^2-x_ix_j )\dd x\quad i, j=1,2,3. $$
Thereby we have the following relation:
\begin{equation}\label{massetc}
m_\eps=\eps^3m_1, \quad \quad  J_\eps=\eps^5J_1.
\end{equation}

\begin{thm}\label{thm-semigroupe}
For every $1<q<\infty$, the semi-group $\tline^q_\eps(t)$ on 
$X_\eps^q$ satisfies the following decay estimates:
	
$\bullet$  For $1 < q \leq p < \infty,$ there exists $C_1(p,q)>0$ 
such that
\begin{equation}\label{LpLq1}
\lVert\tline^q_\eps(t)V^0_\eps\rVert_{X^p_{\eps}} \leq C_1 (p, q)
\m t^{- \frac{3}{2} 
(\frac{1}{q} - \frac{1}{p}  )}\|V^0_\eps\|_{X^q_{\eps}}
\FORALL t>0,\m\m V_\eps^0\in X_\eps^q.
\end{equation}	
	
$\bullet$  For $1 < q < \infty$ and $p = \infty$, there exists 
$C_2(q)>0$ such that
\begin{equation}\label{LpLq2}
\left|\tline^q_\eps(t)V^0_\eps\right| \leq C_2 (q)\m t^{- \frac{3}{2q}}
\|V^0_\eps\|_{X^q_{\eps}}
\FORALL t>0,\m\m V_\eps^0\in X_\eps^q.
\end{equation}	

$\bullet$ For $\frac{3}{2} \leq q<\infty$ and $ q\leq p \leq 
\infty,$ there exists $C_3(p,q)>0$ such that for every $F_\eps 
\in L^{q}(\rline^3; \rline^{3\times3})$ satisfying $F_\eps=0$ 
in $\Sscr^0_\eps$ and $\div F_\eps \in L^r (\rline^3)$ for 
some $r \in (1, p\m]\setminus\{\infty\}$. Then we have
\begin{equation}\label{LpLq3}
\| \tline^r_\eps(t) \pline_\eps \div \, F_\eps \|_{X^{p}_{\eps}} 
\leq C_3 (p, q)\m t^{- \frac{3}{2} (\frac{1}{q} - \frac{1}{p}) - 
\frac{1}{2}} \|F_\eps\|_{L^q( \Fscr^0_\eps)} 
\FORALL t>0.
\end{equation}

$\bullet$ For $1<q\leq p\leq 3$, there exists $C_4(p,q)>0$ such 
that
\begin{equation}\label{LpLq4}
\lVert \nabla\tline^q_\eps V_\eps^0\rVert_{L^p(\Fscr_\eps^0)
}\leq C_4(p,q)\m t^{-\frac{3}{2}\left(\frac{1}{q}-
\frac{1}{p}\right)-\frac{1}{2}} \lVert V_\eps^0\rVert_{X_\eps^q} 
\FORALL t>0,\m\m V_\eps^0\in X^q_\eps.
\end{equation}
\end{thm}

\begin{proof}
As we already mentioned, for fixed $\eps>0$, these results have 
been proved in \cite{ervedoza2022large}. Here we only need to 
verify that the constants $C_1$, $C_2$, $C_3$ and $C_4$ are 
independent of $\eps$. To this aim, we set
$$\tilde x=\frac{x}{\eps},\quad \tilde t=\frac{t}{\eps^2},$$
and
\begin{equation}\label{tildev}
\begin{aligned}
&v_1(\m\tilde t, \tilde x\m):=v_\eps(t,x),\quad \pi_1(\m\tilde 
t, \tilde x\m):=\eps\m \pi_\eps(t,x),\\
&\quad l_{v_1}(\tilde t\m):=l_{v_\eps}(t),\quad \omega_{v_1}
(\tilde t\m):=\eps\m\omega_{v_\eps}(t).
\end{aligned}
\end{equation}
Note that, for every $1<q<\infty$, $V_\eps(t)=\tline_\eps^q(t)
V_\eps^0$ is the solution of the following system:
\begin{equation}\label{abstracthom}
\left\{
\begin{aligned}
&\m \dot V_{\eps}+ A_\eps^q V_{\eps}= 0,\\
&\m V_\eps(0)=V^0_\eps,
\end{aligned}
\right.
\end{equation}
where the operator $A_\eps^q$ has been introduced in \rfb{f-soper} 
and it generates the semi-group $\tline^q_\eps$ on $X_\eps^q$. 
After doing the calculation, together with \rfb{tildev} and 
\rfb{massetc}, we find that $V_1(\tilde t, \tilde x)$ defined by
$$V_1(\tilde t, \tilde x)=v_1(\tilde t, \tilde x)
\mathbbm{1}_{\Fscr_1^0}(\tilde x)+(l_{v_1}(\tilde t\m)+\omega_{v_1}
(\tilde t\m)\times \tilde x)\mathbbm{1}_{\Sscr_1^0}(\tilde x). $$
satisfying \rfb{abstracthom} with $\eps=1$. This implies that 
the system \rfb{abstracthom} is $\eps$-invariant and, in particular, 
that $V_1(\tilde t,\tilde x)=\tline_1(\tilde t\m)V_1^0$ is the 
solution of \rfb{abstracthom} with $\eps=1$. In this case, 
according to \cite[Theorem 7.1]{ervedoza2022large}, we have the 
estimate
$$\lVert\tline^q_1(t)V^0_1\rVert_{X^p_{1}} \leq C_1 (p, q) \m t^{- \frac{3}{2} 
(\frac{1}{q} - \frac{1}{p}  )}\|V^0_1\|_{X^q_{1}}
\FORALL t>0,\m\m V_1^0\in X_1^q. $$
Based on the above analysis, this is equivalent to \rfb{LpLq1}.

The estimate \rfb{LpLq2} can be verified in a similar way. For 
\rfb{LpLq3}, we define
$$F_1(\tilde x):=\frac{1}{\eps} F_\eps(x), $$
then $V_\eps(t)=\tline_\eps^r\mathbb{P}_\eps\div F_\eps$ also 
satisfies \rfb{abstracthom} with $V_\eps^0=\mathbb{P}_\eps\div F_\eps$. 
Hence, according to \cite[Lemma 8.1]{ervedoza2022large}, we also 
have \rfb{LpLq3} where $C_3$ only depends on $p$ and $q$.

For the estimate \rfb{LpLq4}, we employ a duality argument. For 
every $1<q<\infty$, the dual operator of $A_\eps^q$, denoted by 
$(A_\eps^q)^*$, is given by $(A_\eps^q)^*=A_\eps^{q'}$ with 
$\frac{1}{q}+\frac{1}{q'}=1$. For this, please refer to 
\cite[Proposition 5.3]{ervedoza2022large} for more details. We 
denote by $\tline_\eps^{q'}$ the corresponding semi-group generated 
by $A_\eps^{q'}$. Let $F_\eps\in L^{r'}\left(\rline^3;\rline
^{3\times 3}\right)$ satisfy $F_\eps=0$ in $\Sscr_\eps^0$ and 
$\div F_\eps\in L^{q'}(\rline^3)$ with $\frac{3}{2}
\leq r'<\infty$ and $r'\leq q'\leq\infty$. By density, we assume 
that $F_\eps\in C_c^\infty\left(\Fscr_\eps^0; \rline^{3\times 3}
\right)$. For every $V_\eps^0\in X_\eps^q$, we have
\begin{equation*}\label{dual}
\begin{aligned}
\lVert \nabla\tline_\eps^q V_\eps^0\rVert_{L^r(\Fscr_
\eps^0)}
&=\sup_{\lVert F_\eps\rVert_{L^{r'}}\leq 1}\left\langle 
\nabla\tline_\eps^q(t)V_\eps^0, \m F_\eps\right\rangle_{L^r, L^{r'}}\\
&=\sup_{\lVert F_\eps\rVert_{L^{r'}}\leq 1}\int_{\Fscr_\eps^0}
\nabla\tline_\eps^q V_\eps^0: F_\eps\\
&=-\sup_{\lVert F_\eps\rVert_{L^{r'}}\leq 1}\left\langle\tline_
\eps^q V_\eps^0,\m \mathbb{P}_\eps\div F_\eps\right\rangle_
{X_\eps^q, X_\eps^{q'}}\\
&=-\sup_{\lVert F_\eps\rVert_{L^{r'}}\leq 1}\left\langle V_\eps^0,
\m \tline_\eps^{q'}\mathbb{P}_\eps\div F_\eps\right\rangle_{X_\eps^q, 
X_\eps^{q'}}\\
&\leq \lVert V_\eps^0\rVert_{X_\eps^q}\sup_{\lVert F_\eps
\rVert_{L^{r'}}\leq 1}\lVert \tline_\eps^{q'}\mathbb{P}_\eps\div 
F_\eps\rVert_{X_\eps^{q'}},
\end{aligned}
\end{equation*}
where we used the fact that $F_\eps=0$ in $\Sscr_\eps^0$ and 
$F_\eps\cdot n=0$ on $\prt \Sscr_\eps^0$. By using the estimate 
\rfb{LpLq3} in the case that $r=p$, we obtain that
$$\lVert \nabla\tline_\eps^{q}V_\eps^0\rVert_{L^r(\Fscr_\eps^0)}\leq C_3(q', r')\m t^{-\frac{3}{2}\left(\frac{1}{r'}-
\frac{1}{q'}\right)-\frac{1}{2}}\lVert V_\eps^0\rVert_{X_\eps^q}. $$
Recalling the assumption on $r'$ and $q'$, we derive that 
$1<q\leq r\leq 3$. Note that $\frac{1}{r'}-\frac{1}{q'}=
\frac{1}{q}-\frac{1}{r}$, we immediately obtain the estimate 
\rfb{LpLq4}, which ends the proof.
\end{proof}

\begin{rmk}
{\rm During the proof of the last estimate \rfb{LpLq4}, we 
could say that the operator $\nabla\tline_\eps^q$ is the dual 
operator of $\tline_\eps^{q'}\mathbb{P}_\eps\div$ with respect 
to functions $F_\eps$ which vanishes at the boundary 
$\prt\Sscr_\eps^0$.
}
\end{rmk}

\begin{rmk}
{\rm It is worthwhile noting that, for fixed $\eps >0$, the estimates \rfb{LpLq1}-\rfb{LpLq4} in Theorem \ref{thm-semigroupe} hold for arbitrary shape of the rigid body since it is only related to the linearized fluid-body system.
}
\end{rmk}

\subsection{Uniform Estimate }\label{subsection-unif3}
 The main aim of this subsection is to obtain a uniform estimate 
 of $\dot h_\eps (t)$. To do this, we prove the existence and 
 uniqueness of the mild solution for the abstract differential 
 equation \rfb{abstract}-\rfb{eq:Feps}, which is equivalently 
 formulated from the fluid-body system \eqref{eq01}--\eqref{eq06}.
 
 The main idea is to show that \rfb{mild} is indeed the solution 
 of \rfb{abstract}-\rfb{eq:Feps}. Actually, this has been proved 
 in \cite[Theorem 8.2]{ervedoza2022large} by using the fixed-point 
 argument. Here, for completeness, we present again its statement 
 and proof. In particular,  we describe how small the initial data 
 we need and give the precise upper bound for the norm of the 
 solid velocity, which are not specified in \cite{ervedoza2022large}.
 
 Based on the extension of $v_\eps$ in \rfb{extendv}, for every 
 $t>0$, we identify $V_\eps$ with the triple $\begin{bmatrix} v_\eps 
 & l_{v_\eps} & \omega_{v_\eps}\end{bmatrix}^\intercal$. In particular, 
 $V_\eps=\begin{bmatrix} v_\eps & l_{v_\eps} & \omega_{v_\eps}
 \end{bmatrix}^\intercal$ is a mild solution of \rfb{abstract} if 
 and only if it satisfies \rfb{mild}. 
 
\begin{prop}\label{fixedpoint}
There exists $\gamma_0$, $\mu_0>0$ independent of $\eps$ such that 
for every $V_\eps^0\in X_\eps^2\cap X_\eps^3$ with 
\begin{equation*}\label{initialcond}
\lVert V_\eps^0\rVert_{X_\eps^2\cap X_\eps^3}\leq \gamma_0,
\end{equation*}
there exists a unique solution $V_\eps$ satisfying \rfb{mild} and
\begin{equation}\label{spacesolu}
\sup_{t>0} t^{\frac{3}{2}\left(\frac{1}{2}-\frac{1}{p}\right)}
\lVert V_\eps\rVert_{X_\eps^p}\leq \mu_0
\end{equation}
for every $2\leq p\leq\infty$.
\end{prop}

\begin{proof}
To present the proof clearly, we divide it into the following 
two steps.

{\bf Step 1:} {\em The existence and uniqueness of $V_\eps$ with $V_\eps^0\in 
X_\eps^3$. } The main idea is to use Banach’s contraction principle. 
We first define the space 
\begin{equation*}
\Kscr_\eps:=\begin{Bmatrix}V_\eps=\begin{bmatrix} v_\eps \\ 
l_{v_\eps} \\ \omega_{v_\eps} \end{bmatrix} \left|\m \m\m 
t^{\frac{1}{4}}V_\eps\in C^0\left((0,\infty); X_\eps^6\right), 
\quad t^{\frac{1}{2}} V_\eps\in C^0\left((0,\infty); 
X_\eps^\infty\right)\right.\\
\quad \text{and}\quad \min\left\{t^{\frac{1}{2}}, 1\right\}
\nabla v_\eps\in C^0\left((0,\infty); L^3(\Fscr_\eps^0)
\right) \end{Bmatrix},
\end{equation*}
endowed with the norm
\begin{equation*}
\begin{aligned}
\lVert V_\eps\rVert_{\Kscr_\eps} :=\lVert\m t^{\frac{1}{4}} 
V_\eps(t)\rVert_{L^\infty\left(0,\infty; X_\eps^6\right)}&+
\lVert \m t^{\frac{1}{2}} V_\eps(t)\rVert_{L^\infty
\left(0,\infty; X_\eps^\infty\right)}\\
&+\lVert\m \min\{t^{\frac{1}{2}}, 1\}\nabla v_\eps\rVert_
{L^\infty\left(0,\infty; L^3(\Fscr_\eps^0)\right)}.
\end{aligned}
\end{equation*}
Now, according to the Duhamel formulation \eqref{mild}, we introduce the map 
$$
\Lambda_\eps: \Kscr_\eps\longrightarrow \Kscr_\eps,$$
with 
$$
\Lambda_\eps(V_\eps)(t):=\tline_\eps(t)V_\eps^0+\int_0^t\tline_
\eps(t-s)\mathbb{P}_\eps \div F_\eps(v_\eps(s))\dd s \FORALL t\geq0,
$$
where $F_\eps(v_\eps)$ has been introduced in \rfb{eq:Feps}. 
We further introduce
$$ \Phi(V_\eps, W_\eps)(t)=\int_0^t\tline_\eps(t-s)\mathbb{P}_\eps 
\div G_\eps(v_\eps, w_\eps)(s)\dd s \FORALL t\geq0,$$
with 
\begin{equation}\label{G}
G_\eps(v_\eps, w_\eps)=
\left\{
	\begin{aligned}
		&v_\eps\otimes (\ell_{w_\eps}- w_\eps) & \text{ 
  on $\Fscr_\eps^0$},\,\\
		&0  &\text{ on $\Sscr_\eps^0$}.
	\end{aligned}\right.   
\end{equation}
Note that for $v_\eps$, $w_\eps$ satisfy the conditions in 
$\Kscr_\eps$, it is not difficult to see that $\div G_\eps(v_\eps, 
w_\eps)\in L^3(\rline^3)$ for fixed $s>0$. This 
implies that we could use the estimate \rfb{LpLq3} for 
$3\leq p\leq \infty $. In the remaining part of the proof, we 
use the notation $\mathcal{B}(\cdot, \cdot)$ to represent the function:
$$\mathcal{B}(\alpha, \beta):=\int_0^1(1-s)^{-\alpha} s^{-\beta}\dd s. $$
By change of variable, we notice that for all $t >0$, 
$$ \mathcal{B}\left(\alpha, \beta\right) = t^{\alpha + \beta -1}\int_0^t(t-s)^{- \alpha}s^{- \beta}
\dd s.$$
By using H\"older inequality, we do the following calculation:
\begin{equation}\label{1st}
\begin{aligned}
t^{\frac{1}{4}}\lVert\Phi(v_\eps, w_\eps)\rVert_{X_\eps^6}
&\leq t^{\frac{1}{4}} C_3(6,6)\int_0^t(t-s)^{\frac{1}{2}}\left(
\lVert V_\eps\rVert_{X_\eps^\infty}\lVert W_\eps\rVert_{X_\eps^6}
+|l_{w_\eps}|\lVert W_\eps\rVert_{X_\eps^6}\right)\dd s\\
&\leq 2\m C_3(6,6)\lVert V_\eps\rVert_{\Kscr_\eps}\lVert 
W_\eps\rVert_{\Kscr_\eps}\m \mathcal{B} \left(\frac{1}{2}, \frac{3}{4}\right),
\end{aligned}
\end{equation}
where we used the estimate \rfb{LpLq3} for $p=6=q$.

Similarly, we have
\begin{equation}\label{2nd}
\begin{aligned}
t^{\frac{1}{2}}\lVert\Phi(v_\eps, w_\eps)\rVert_{X_\eps^\infty}
&\leq 2\m C_3(\infty, 6)\lVert V_\eps\rVert_{\Kscr_\eps}\lVert 
W_\eps\rVert_{\Kscr_\eps} \mathcal{B}\left(\frac{3}{4}, \frac{3}{4}\right),
\end{aligned}
\end{equation}
where we used the estimate \rfb{LpLq3} with $p=\infty$ and $q=6$.

For the last constrained condition in $\Kscr_\eps$, we also have
\begin{equation}\label{3rd-}
\begin{aligned}
&\min\{t^{\frac{1}{2}}, 1\}\lVert \nabla\Phi(v_\eps, w_\eps)
\rVert_{L^3(\Fscr_\eps^0)}\\
\leq& \min\{t^{\frac{1}{2}}, 1\}\int_0^t\m C_4(3,2)(t-s)^{-\frac{3}{4}}
\lVert W_\eps\rVert_{X_\eps^6}\lVert \nabla v_\eps\rVert_{L^3(\Fscr_\eps^0)}\m\dd s\\
&\quad+\min\{t^{\frac{1}{2}}, 1\}\int_0^t C_4(3,3)(t-s)^{
-\frac{1}{2}}|l_{w_\eps}|\lVert\nabla v_\eps\rVert_{L^3(
\Fscr_\eps^0)}\m\dd s\\
\leq &\m C_4(3,2)\lVert V_\eps\rVert_{\Kscr_\eps} \lVert W_\eps
\rVert_{\Kscr_\eps}\int_0^t(t-s)^{-\frac{3}{4}}s^{-\frac{1}{4}}
\frac{\min\{t^{\frac{1}{2}}, 1\}}{\min\{s^{\frac{1}{2}}, 1\}}\dd s\\
&\quad+C_4(3,3)\lVert V_\eps\rVert_{\Kscr_\eps} \lVert W_\eps
\rVert_{\Kscr_\eps} \int_0^t(t-s)^{-\frac{1}{2}}\frac{
\min\{t^{\frac{1}{2}}, 1\}}{\min\{s^{\frac{3}{4}}, s^{\frac{1}{2}}\}}\dd s,
\end{aligned}
\end{equation}
where we used the estimate \rfb{LpLq4} with $p=3$ and $q=2$, $q=3$. 
In particular, for $W_\eps\in \Kscr_\eps$, we used
$$|l_{w_\eps}|\leq \frac{1}{\max\{s^{\frac{1}{4}}, s^{\frac{1}{2}}\}}
\lVert W_\eps\rVert_{\Kscr_\eps}.  $$
and
$$\max\{s^{\frac{1}{4}}, s^{\frac{1}{2}}\}\min\{s^{\frac{1}{2}}, 
1\}\geq\min\{s^{\frac{3}{4}}, s^{\frac{1}{2}}\} .$$
Note that for every $t\in(0,1)$ we have
$$\int_0^t \frac{\min\{t^{\frac{1}{2}}, 1\}}{(t-s)^{\frac{3}{4}}
s^{\frac{1}{4}}\min\{s^{\frac{1}{2}}, 1\}}\dd s
= \int_0^t\frac{t^{\frac{1}{2}}}{(t-s)^{\frac{3}{4}}
s^{\frac{3}{4}}}\dd s = \mathcal{B}\left(\frac{3}{4}, 
\frac{3}{4}\right),$$
$$\int_0^t\frac{\min\{t^{\frac{1}{2}}, 1\}}{(t-s)^{\frac{1}{2}}
\min\{s^{\frac{3}{4}}, s^{\frac{1}{2}}\}} \dd s
= \int_0^t\frac{t^{\frac{1}{2}}}{(t-s)^{\frac{1}{2}}
s^{\frac{3}{4}}} \dd s 
=t^{\frac{1}{4}}\mathcal{B}\left(\frac{1}{2}, \frac{3}{4}\right)\leq \mathcal{B}
\left(\frac{1}{2}, \frac{3}{4}\right).$$
For $t\geq 1$, we have
\begin{equation*}
\begin{aligned}
\int_0^t\frac{\min\{t^{\frac{1}{2}}, 1\}}{(t-s)^{\frac{3}{4}}
\min\{s^{\frac{3}{4}}, s^{\frac{1}{4}}\}}\dd s
\leq & \mathcal{B}\left(\frac{3}{4}, \frac{3}{4}\right)+\mathcal{B} \left(\frac{3}{4}, 
\frac{1}{4}\right), 
\end{aligned}
\end{equation*}
\begin{equation*}
\begin{aligned}
\int_0^t\frac{\min\{t^{\frac{1}{2}}, 1\}}{(t-s)^{\frac{1}{2}}
\min\{s^{\frac{3}{4}}, s^{\frac{1}{2}}\}}\dd s
\leq &
\mathcal{B}\left(\frac{1}{2}, \frac{3}{4}\right)+\mathcal{B}\left(\frac{1}{2}, 
\frac{1}{2}\right).
\end{aligned}
\end{equation*}
Based on the above calculation, we obtain from \rfb{3rd-} that
\begin{equation}\label{3rd}
\min\{t^{\frac{1}{2}}, 1\}\lVert \nabla\Phi(v_\eps, w_\eps)
\rVert_{L^3(\Fscr_\eps^0)}\\
\leq \tilde C\lVert V_\eps\rVert_{\Kscr_\eps}\lVert W_\eps
\rVert_{\Kscr_\eps},
\end{equation}
with $$\tilde C=C_4(3,2)\left( \mathcal{B}\left(\frac{3}{4}, \frac{3}{4}
\right)+\mathcal{B}\left(\frac{3}{4}, \frac{1}{4}\right)\right)+C_4(3,3)
\left(\mathcal{B}\left(\frac{1}{2}, \frac{3}{4}\right)+\mathcal{B}\left(\frac{1}{2}, 
\frac{1}{2}\right)\right). $$

Now putting together \rfb{1st}-\rfb{3rd}, we have
\begin{equation}\label{Dcons}
\lVert\Phi(v_\eps, w_\eps)\rVert_{\Kscr_\eps}\leq D_0\lVert 
V_\eps\rVert_{\Kscr_\eps}\lVert W_\eps\rVert_{\Kscr_\eps},
\end{equation}
where the constant $D_0$ is
$$ D_0=2C_3(6,6)\mathcal{B}\left(\frac{1}{2}, \frac{3}{4}\right)+2C_3(
\infty, 6)\mathcal{B}\left(\frac{3}{4}, \frac{3}{4}\right)+\tilde C.
$$

Now for $\lVert \tline_\eps(t) V_\eps^0\rVert_{\Kscr_\eps}$, 
we have estimate
$$ t^{\frac{1}{4}}\lVert \tline_\eps(t) V_\eps^0\rVert_{X_\eps^6}
\leq C_1(6,3)\lVert V_\eps^0\rVert_{X_\eps^3},\quad
t^{\frac{1}{2}}\lVert\tline_\eps(t) V_\eps^0\rVert_{X_\eps^\infty}
\leq C_2(3)\lVert V_\eps^0\rVert_{X_\eps^3}, $$
and 
\begin{multline*}
\min\{t^{\frac{1}{2}}, 1\}\lVert\nabla \tline_\eps(t) V_\eps^0
\rVert_{L^3(\Fscr_\eps^0)}
\leq \min\{t^{\frac{1}{2}}, 1\} C_4(3,3)t^{-\frac{1}{2}}\lVert 
V_\eps^0\rVert_{X_\eps^3}\\
\leq \left\{
\begin{aligned}
&C_4(3,3)\lVert V_\eps^0\rVert_{X_\eps^3}\quad (0<t<1)\\
&t^{-\frac{1}{2}} C_4(3,3)\lVert V_\eps^0\rVert_{X_\eps^3}\quad 
(t\geq 1)
\end{aligned}\right.
\leq C_4(3,3)\lVert V_\eps^0\rVert_{X_\eps^3}.
\end{multline*}
Hence, we obtain
\begin{equation}\label{tline}
\lVert \tline_\eps(t) V_\eps^0\rVert_{\Kscr_\eps}\leq D_1\lVert 
V_\eps^0\rVert_{X_\eps^3},
\end{equation}
where $$ D_1=C_1(6,3)+C_2(3)+C_4(3,3).$$

Combining with \rfb{Dcons} and \rfb{tline} we arrive at
\begin{equation}\label{Lambdaes}
\lVert\Lambda_\eps (V_\eps)\rVert_{\Kscr_\eps}\leq D_1\lVert 
V_\eps^0\rVert_{X_\eps^3}+D_0\lVert V_\eps\rVert^2_{\Kscr_\eps}.
\end{equation}
Setting
\begin{equation*}\label{Rlambda}
R:=\frac{1}{4D_0}, \quad\lambda_0:=\min\left\{\frac{R}{2D_1}, 
R\right\}.
\end{equation*} 
Let $\lVert V_\eps^0\rVert_{X_\eps^3}\leq \lambda_0$. Assuming 
that $\lVert V_\eps\rVert_{\Kscr_\eps}\leq R$ we derive from 
\rfb{Lambdaes} that
$$\lVert \Lambda_\eps(V_\eps)\rVert_{\Kscr_\eps}\leq \frac{R}{2}
+\frac{R}{4}\leq R,  $$
which gives that the closed ball $B_{\Kscr_\eps}(0, R):=\{V_\eps
\in \Kscr_\eps\m|\m \lVert V_\eps\rVert_{\Kscr_\eps}\leq R \}$ 
is $\Lambda_\eps$-invariant. With the same initial data, for $V_\eps$, $W_\eps\in 
B_{\Kscr_\eps}(0, R)$,  taking 
\rfb{Dcons} into account, we compute
\begin{equation*}
\begin{aligned}
\lVert \Lambda_\eps(V_\eps)-\Lambda_\eps(W_\eps)\rVert_{\Kscr_\eps}
&=\lVert \Phi(v_\eps, v_\eps-w_\eps)+\Phi(v_\eps-w_\eps, w_\eps)
\rVert_{\Kscr_\eps} \\
&\leq 2D_0 R\lVert V_\eps-W_\eps\rVert_{\Kscr_\eps}\leq \frac{1}{2}
\lVert V_\eps-W_\eps\rVert_{\Kscr_\eps},
\end{aligned}
\end{equation*}
which implies that $\Lambda_\eps$ is a contraction mapping. 
According to Banach fixed-point theorem, there exists a unique 
$V_\eps\in \Kscr_\eps$, such that $\Lambda_\eps(V_\eps)=V_\eps$.
Moreover, from \rfb{Lambdaes} we find that there exists $C_0$, 
such that
\begin{equation}\label{boundsol}
\lVert V_\eps\rVert_{\Kscr_\eps}\leq C_0 \lVert V_\eps^0\rVert_{X_\eps^3}.
\end{equation}

{\bf Step 2:} {\em Further regularity of $V_\eps$ when $V_\eps^0
\in X_\eps^2\cap X_\eps^3$.} For $V_\eps^0\in X_\eps^2\cap X_\eps^3$, 
we shall show that the solution $V_\eps$ obtained from Step 1 
satisfies additional regularity. Now we reconsider the function 
$G_\eps(v_\eps, v_\eps)$ defined in \rfb{G} and check the 
availability of the estimate \rfb{LpLq3}. Before going this, we 
first show in this case that $\nabla v_\eps$ have new property:
\begin{equation}
\begin{aligned}
&\min\{t^{\frac{1}{2}}, 1\}\lVert\nabla v_\eps\rVert_{L^2
(\Fscr_\eps^0)}\\
\leq&\min\{t^{\frac{1}{2}}, 1\} \left(\lVert\nabla
\tline_\eps V_\eps^0\rVert_{L^2}+\lVert\nabla\Phi(v_\eps, 
v_\eps)\rVert_{L^2}\right)\\
\leq & C_4(2,2)\lVert V_\eps^0\rVert_{X_\eps^2}+2 C_4(2,2)C_0
\lVert V_\eps^0\rVert_{X_\eps^3}\mathcal{B}\left(\frac{1}{2}, \frac{1}{2}
\right)\min\{t^{\frac{1}{2}}, 1\}\lVert\nabla v_\eps\rVert_{L^2},
\end{aligned}
\end{equation}
where we used \rfb{boundsol}. 
From above, we see that if $V_\eps^0$ satisfies
$$\lVert V_\eps^0\rVert_{X_\eps^3}\leq \min\left\{\lambda_0,
\m\m \frac{1}{2\m C_0\m C_4(2,2)\mathcal{B}\left(\frac{1}{2}, \frac{1}{2}
\right)}\right\},  $$
we derive from \rfb{nablav} that $\min\{t^{\frac{1}{2}}, 1\}
\lVert\nabla v_\eps\rVert_{L^\infty\left(L^2(\Fscr_\eps^0)
\right)}<\infty$. Based on this condition, it 
is not difficult to check that $\div G_\eps(v_\eps, v_\eps)\in 
L^2(\rline^3)$, which means that the estimate 
\rfb{LpLq3} makes sense when $p\geq 2$. 

By using \rfb{LpLq1} and \rfb{LpLq3}, we do some calculations 
as follows:
\begin{equation*}
\begin{aligned}
\lVert V_\eps\rVert_{X_\eps^2}
&\leq C_1(2,2)\lVert V_\eps^0\rVert_{X_\eps^2}+C_3(2,2)\int_0^t
(t-s)^{-\frac{1}{2}}\left(|l_{v_\eps}|\lVert V_\eps\rVert_{X_\eps^2}
+\lVert V_\eps\rVert_{X_\eps^\infty}\lVert V_\eps\rVert_{X_\eps^2}\right)\\
&\leq C_1(2,2)\lVert V_\eps^0\rVert_{X_\eps^2}+2\m C_0\m C_3(2,2)
\lVert V_\eps^0\rVert_{X_\eps^3}\mathcal{B}\left(\frac{1}{2}, \frac{1}{2}
\right)\lVert V_\eps\rVert_{X_\eps^2}.
\end{aligned}
\end{equation*}
Thus, we assume that $V_\eps^0$ satisfies
$$\lVert V_\eps^0\rVert_{X_\eps^3}\leq \min\left\{\lambda_0,
\m\m \frac{1}{2\m C_0\m \mathcal{B}\left(\frac{1}{2}, \frac{1}{2}\right)
\max\{C_4(2,2), C_3(2,2)\}}\right\}:=\gamma_0,  $$
then $V_\eps\in L^\infty\left(0,\infty; X_\eps^2\right)$. 

Moreover, we also have
\begin{equation*}
\begin{aligned}
t^{\frac{3}{4}}\lVert V_\eps\rVert_{X_\eps^\infty}\leq C_2(2)
\lVert V_\eps^0\rVert_{X_\eps^2}+ 2C_3(\infty, 2)\mathcal{B}
\left(\frac{5}{4}, \frac{1}{2}\right)\lVert V_\eps
\rVert_{\Kscr_\eps}\lVert V_\eps\rVert_{X_\eps^2}<\infty.
\end{aligned}
\end{equation*}
By using interpolation inequality, we obtain that 
$t^{\frac{3}{2}\left(\frac{1}{2}-\frac{1}{p}\right)} V_\eps
\in L^\infty(0,\infty; X_\eps^p),$
for every $2\leq p\leq \infty$.
More precisely, with $\lVert V_\eps^0\rVert_{X_\eps^2\cap X_\eps^3}
\leq \gamma_0 $, we conclude that \rfb{spacesolu}
holds with
$$\mu_0= C_2(2)\m \gamma_0+\frac{2\m \gamma_0^2\m C_3(\infty, 2)
\m C_1(2,2)\mathcal{B}\left(\frac{5}{4}, \frac{1}{2}\right)}{1-2\gamma_0\m 
C_0\m C_3(2,2)\mathcal{B}\left(\frac{1}{2}, \frac{1}{2}\right)}.$$
Now the proof is completed.
\end{proof}

\begin{rmk}
{\rm We remark that the mild solution $V_\eps=\begin{bmatrix} 
v_\eps & l_{v_\eps} & \omega_{v_\eps}
 \end{bmatrix}^\intercal$ obtained in 
Proposition \ref{fixedpoint} is a weak solution 
of the fluid-body system \eqref{governing}, associated with the 
initial data $(v_\eps^0,\m  l_\eps^0,\m \omega_\eps^0)$, in the 
sense of Definition \ref{defv}. The proof is similar to the proof 
of the similar results in two-dimensional case (see for instance 
\cite[proof of Proposition 2.5]{takahashi2004global}) and also 
similar for the Navier-stokes system (see 
 \cite{1972FJR} and \cite{gilles2018navier}), 
after the appropriate choice of the function spaces.
}
\end{rmk}

Based on the structure of $V_\eps$ in the space $\Kscr_\eps$, 
we immediately obtain from \rfb{spacesolu} that the solid 
velocity $l_{v_\eps}=\dot h_\eps(t)$ satisfies 
$$\sup_{t>0} t^{\frac{3}{4}}|\dot h_\eps(t)|\leq \mu_0.$$
Recalling that 
$h_\eps(0)=0$, according to Leibniz formula, we have for 
every $t\in [0,T]$,
$$|h_\eps(t)|\leq \int_0^t|\dot h_\eps(\tau)|\dd \tau\leq 
\mu_0\int_0^t\tau^{-\frac{3}{4}}\dd \tau\leq 4\m\mu_0 T^{\frac{1}{4}},
$$
which means that the sequence $(h_\eps)$ is uniformly bounded.
Moreover, we see that
$(h_\eps)$ is bounded in $W^{1,q}(0, T; \rline^3)$ for 
$1<q<\frac{4}{3}$. This implies that, up to a subsequence,
$$h_\eps\rightharpoonup h \quad \text{weakly in}\quad W^{1,q}(0, T).$$
According to Morrey's inequality, we have $W^{1,q}(0, T)
\hookrightarrow C^{0,\alpha}[0, T]$ with $0<\alpha<1-\frac{1}{q}$ 
and $1<q<\frac{4}{3}$, which gives that $(h_\eps)$ is uniformly 
equi-continuous. Therefore, according to the
Arzel\`a-Ascoli theorem and the compact embedding 
$C^{0,\alpha}[0, T]\hookrightarrow C^0[0, T]$, we obtain that 
$(h_{\eps})$ converges uniformly up 
to a subsequence, i.e. 
$$h_{\eps}\toep h \quad \text{uniformly in}\quad [0,T], $$
where $h\in W^{1,q}(0, T)$ with $1<q<\frac{4}{3}$.
We used the uniqueness of the limit in the above analysis.

\section{Modified test function}\label{section-testfunc}

In this section we construct an approximate sequence of the 
test functions for the fluid-body system. This will be used to 
justify the limit of the solution in the remaining part. 
The main idea comes from the technique used in 
\cite{lacave2017small}.

We first define the cut-off function $\chi(x):\rline^3\longrightarrow 
[0,1]$ and $\chi\in C^\infty(\rline^3; [0,1])$ such that 
\begin{equation}\label{chi1}
\chi(x)=\left\{
\begin{aligned}
&0 \quad \text{in}\quad B\left(0,\frac{3}{2}\right),\\
&1\quad \text{in}\quad \left(B\left(0, 2\right)\right)^c.
\end{aligned}\right.
\end{equation}
Denote the annulus $\Omega_1:=B\left(0, 2\right)\setminus B
\left(0, \frac{3}{2} \right)$. Using \rfb{chi1}, we introduce 
a cut-off function $\chi_\eps(t,x)$ near the ball $B(h(t), \eps)$, for every $t>0$, as follows:
\begin{equation}\label{chi}
\chi_\eps(t,x):=\chi\left(\frac{x-h(t)}{\eps}\right)=
\left\{
\begin{aligned}
&0 \quad \text{in}\quad B\left(h(t),\frac{3}{2}\eps\right),\\
&1\quad \text{in}\quad \left(B\left(h(t), 2\eps\right)\right)^c.
\end{aligned}
\right.
\end{equation}
Similarly, we denote the annulus for $\chi_\eps$ as $\Om_\eps (t):
=B\left(h(t), 2\eps\right)\setminus B\left(h(t), \frac{3}{2}\eps\right)$. We observe that $\text{supp} (\chi_\epsilon - 1) \subseteq B\left(h(t), 2 \eps\right)$ and $\text{supp} (\nabla\chi_\eps) \subseteq \Om_\eps (t)$.
We introduce the following lemma 
for some properties of the cut-off function $\chi_\eps(t,x)$, which are frequently used in the convergence of the test function.

\begin{lem}\label{lemmachi}
For every $T>0$, the cut-off function $\chi_\eps(t,x)$ defined in \rfb{chi} satisfies
\begin{itemize}
\item $\chi_\eps\in W^{1,q}(0,T; C^\infty(\rline^3))$ for 
$1<q<\frac{4}{3}$;
\item For every $t\in (0, T)$, $\chi_\eps$ vanishes in the ball 
$B\left(h(t), \frac{3}{2}\eps\right)$;
\item There exists $C>0$, such that
$$ \lVert \chi_\eps(t, \cdot)\rVert_{L^\infty(\rline^3)} = 1, \quad
\lVert \chi_\eps(t, \cdot)-1\rVert_{L^2(\rline^3)}\leq C
\eps^{\frac{3}{2}},\quad \lVert \nabla\chi_\eps(t, \cdot)
\rVert_{L^2(\rline^3)}\leq C\eps^{\frac{1}{2}}.$$
\end{itemize}
\begin{proof}
It is clear that the regularity of $\chi_\eps$ for time 
inherits from the regularity of $h(t)$. Based on the definition 
in \rfb{chi}, $\chi_\eps$ is $C^\infty$ with respect to the 
space variable. The second one is obvious from \rfb{chi}.

For the third statement, $ \lVert \chi_\eps(t, \cdot)\rVert_{L^\infty(\rline^3)} = 1
$ is easy to check, so we do some calculation as follows for the last two estimates:
\begin{equation*}
\begin{aligned}
 \lVert \chi_\eps(t,x)-1\rVert_{L^2(\rline^3)}&=\left(\int_
 {B(h(t),2\eps)}\left|\chi\left(\frac{x-h(t)}{\eps}\right)-1
 \right|^2\dd x\right)^{\frac{1}{2}}\\
&=\left(\int_{B(0,2)}\left|\chi(y)-1\right|^2\eps^3\dd y
\right)^{\frac{1}{2}} \leq C\eps^{\frac{3}{2}},
\end{aligned}
\end{equation*}
\begin{equation*}
\begin{aligned}
\lVert\nabla\chi_\eps(t,x)\rVert_{L^2(\rline^3)}&=\left(
\int_{\Om_\eps (t)}\left|\nabla\chi\left(\frac{x-h(t)}{\eps}
\right)\right|^2\dd x\right)^{\frac{1}{2}}\\
&=\left(\int_{\Om_1}\left|\frac{1}{\eps}\nabla\chi(y)\right|
^2\eps^3\dd y\right)^{\frac{1}{2}}\leq C\eps^{\frac{1}{2}}.
\end{aligned}
\end{equation*}
\end{proof}
\end{lem}

\begin{prop}\label{Prop_test-func}
Let $T>0$, assume that $\varphi\in C_c^\infty([0, T)\times\rline^3)$ 
with $\div \varphi=0$. For every $\eps>0$, there exists 
$\varphi_\eps\in W_c^{1,q}([0, T); H^1(\rline^3))$ with 
$1<q<\frac{4}{3}$ satisfying
\begin{equation}\label{divfree}
\div \varphi_\eps=0 \quad \text{in}\quad (0, T)\times \rline^3,
\end{equation}
\begin{equation}\label{inner}
\varphi_\eps\equiv 0\quad t\in [0, T), \quad x\in B\left(h(t), 
\frac{3}{2}\eps\right),
\end{equation}
\begin{equation}\label{convergphi}
\varphi_\eps\toep \varphi\quad \text{strongly in}\quad 
L^\infty(0,T; H^1(\rline^3)),
\end{equation}
and
\begin{equation}\label{timeconv}
\prt_t{\varphi_\eps}\toep \prt_t\varphi\quad \text{strongly in}
\quad L^q(0, T; L^2(\rline^3)).
\end{equation}
\end{prop}

\begin{proof}
With the cut-off $\chi$ defined in \rfb{chi1}, we introduce
\begin{equation*}\label{tildephi}
\tilde \varphi_\eps(t,y):=\varphi(t,h(t)+\eps y)\cdot\nabla\chi(y).
\end{equation*}
It is not difficult to see that $\tilde \varphi_\eps(t,y)\in 
W^{1,q}_c([0, T); L^2(\Omega_1))$ and $\tilde\varphi_\eps$ satisfies 
zero-mean property on $\Omega_1$ ($\Om_1$ has been introduced in \rfb{chi1}). Indeed, we have
\begin{equation*}
\begin{aligned}
\int_{\Om_1}\tilde\varphi_\eps(t,y)\dd y&=\int_{\Om_1}\varphi(t, 
h(t)+\eps y)\cdot\nabla\chi(y)\dd y\\
&=\int_{\prt B(0,2)}\varphi(t,h(t)+\eps y)\cdot n\dd s
=\int_{B(0,2)}\div\varphi(t, h(t)+\eps y)\dd y =0,
\end{aligned}
\end{equation*}
where we used $\div\varphi=0$ and the definition of $\chi$ in 
\rfb{chi1}. In this case, the problem 
\begin{equation*}\label{problem-gtlide}
\left\{
\begin{aligned}
&\m\div \tilde g_\eps=\tilde \varphi_\eps\quad \text{in}\quad 
\Omega_1,\\
&\m\tilde g_\eps=0\quad \text{at}\quad \prt\Omega_1,
\end{aligned}
\right.
\end{equation*}
has a solution $\tilde g_\eps\in W^{1,q}(0, T; H_0^1(\Omega_1))$ 
and there exists $C>0$, such that $\tilde g_\eps(t, y)$ satisfies
\begin{equation}\label{tildeineq1}
\lVert\tilde g_\eps\rVert_{L^\infty(0,T; H^1(\Om_1))}\leq C
\lVert\tilde \varphi_\eps\rVert_{L^\infty(0,T; L^2(\Om_1))},
\vspace{+1.5mm}
\end{equation}
\begin{equation}\label{tildeineq2}
\lVert\prt_t\tilde g_\eps\rVert_{L^q(0, T;H^1(\Om_1)) }
\leq C\lVert\prt_t\tilde \varphi_\eps\rVert_{L^q(0, T; L^2(\Om_1))}.
\end{equation}
For the above inequalities, please refer to, for instance, 
\cite[Theorem III.3.1 and Exercise III.3.6]{galdi2011introduction} 
for more details. 

Consider zero-extension of $\tilde g_\eps$ on $\rline^3$ and define 
\begin{equation}\label{phiepss}
\varphi_\eps(t,x):=\varphi(t,x)\chi_\eps(t,x)
-g_\eps(t,x),
\end{equation}
with 
$$g_\eps(t,x):=\tilde g_\eps\left(t, \frac{x-h(t)}{\eps}\right),  $$
where the function $\chi_\eps$ has been introduced in \rfb{chi}.
After doing change of variables $y=\frac{x-h(t)}{\eps}$, we find that 
$$\lVert\tilde g_\eps(t,y)\rVert_{L^2(\rline^3)}=
\eps^{-\frac{3}{2}}\lVert g_\eps(t,x)\rVert_{L^2(\rline^3)}, $$
$$\lVert\nabla_y\tilde g_\eps(t,y)\rVert_{L^2(\rline^3)}=
\eps^{-\frac{1}{2}}\lVert \nabla_x g_\eps(t,x)\rVert_{L^2(\rline^3)}.$$
Moreover, we have
$$ \lVert\tilde \varphi_\eps(t,y)\rVert_{L^2(\Om_1)}\leq 
\lVert\varphi\rVert_{L^\infty(\Om_1)}\lVert\nabla\chi\rVert_{L^2(\Om_1)}
\leq C\lVert\varphi\rVert_{L^\infty(\rline^3)}.$$
We obtain from \rfb{tildeineq1} that
\begin{equation}\label{ineqvarphi}
\eps^{-1}\lVert g_\eps\rVert_{L^\infty(0, T;L^2(\rline^3))}
+\lVert\nabla g_\eps\rVert_{L^\infty(0, T; L^2(\rline^3))}
\leq C\eps^{\frac{1}{2}}\lVert\varphi\rVert_{L^\infty((0, T)
\times \rline^3)}.
\end{equation}

Now, we consider the difference $\varphi_\eps-\varphi$. 
Using the formula of $\varphi_\eps$ in \rfb{phiepss} and 
the inequality \rfb{ineqvarphi}, we have
\begin{equation*}
\begin{aligned}
&\eps^{-1}\lVert\varphi_\eps-\varphi\rVert_{L^\infty(0, T; 
L^2(\rline^3))}+\left\lVert\nabla\varphi_\eps-\nabla\varphi
\right\rVert_{L^\infty(0, T; L^2(\rline^3))}\\
&\leq C\eps^{\frac{1}{2}}\lVert\varphi\rVert_{L^\infty((0, 
T)\times \rline^3)}+\eps^{-1}\left\lVert\varphi\left(\chi_\eps
-1\right)\right\rVert_{L^\infty(0, T; L^2(\rline^3))}\\
&\qquad +\left\lVert\nabla\left(\varphi\chi_\eps-\varphi\right)
\right\rVert_{L^\infty(0, T; L^2(\rline^3))}.
\end{aligned}
\end{equation*}
By using the properties in Lemma \ref{lemmachi}, we estimate 
the last two terms on the above expression:
$$ \eps^{-1}\lVert\varphi(\chi_\eps-1)\rVert_{L^\infty(0, T; 
L^2(\rline^3))}\leq C\eps^{\frac{1}{2}}\lVert\varphi\rVert_
{L^\infty((0, T)\times \rline^3)},$$
\begin{equation*}
\begin{aligned}
\left\lVert\nabla\left(\varphi\chi_\eps-\varphi\right)\right
\rVert_{L^\infty(0, T; L^2(\rline^3))} &=\lVert\nabla\varphi
(\chi_\eps-1)+\varphi\nabla\chi_\eps\rVert_{L^\infty(0, T; 
L^2(\rline^3))}\\
&\leq C\eps^{\frac{3}{2}}\lVert \nabla\varphi\rVert_{L^\infty
((0, T)\times \rline^3)}+C\eps^{\frac{1}{2}}\lVert\varphi
\rVert_{L^\infty((0, T)\times \rline^3)}\\
&\leq C\eps^{\frac{1}{2}}\lVert\varphi\rVert_{W^{1, \infty}
((0, T)\times \rline^3)}.
\end{aligned}
\end{equation*}
Hence, we conclude that
\begin{equation*}\label{differencevarphi}
\eps^{-1}\lVert\varphi_\eps-\varphi\rVert_{L^\infty(0, T; 
L^2(\rline^3))}+\left\lVert\nabla\varphi_\eps-\nabla\varphi
\right\rVert_{L^\infty(0, T; L^2(\rline^3))}
\leq C\eps^{\frac{1}{2}}\lVert \varphi\rVert_{W^{1, \infty}
((0, T)\times \rline^3)},
\end{equation*}
which implies directly that 
$$\varphi_\eps\longrightarrow \varphi \quad \text{strongly 
in }\quad L^\infty(0, T; H^1(\rline^3)).$$

For the time derivative, using \rfb{phiepss} we compute:
\begin{equation}\label{timeprt}
\begin{aligned}
\prt_t\varphi_\eps-\prt_t\varphi&=\chi_\eps\prt_t\varphi+
\varphi\nabla\chi_\eps-\prt_t g_\eps(t,x)+\nabla\tilde 
g_\eps(t,y)\left(-\frac{1}{\eps}\dot h(t)\right)-\prt_t\varphi\\
&=\prt_t\varphi\m(\chi_\eps-1)+\varphi\nabla\chi_\eps-\dot 
h(t)\nabla g_\eps(t,x)-\prt_t g_\eps\left(t,x\right).
\end{aligned}
\end{equation}
We estimate the terms on the right hand side of \rfb{timeprt} 
one by one:
$$ \lVert\prt_t\varphi(\chi_\eps-1)\rVert_{L^2(\rline^3)}\leq \lVert\prt_t\varphi\rVert_{L^\infty(\rline^3)}\lVert\chi_\eps-
1\rVert_{L^2(\rline^3)}\leq C\eps^{\frac{3}{2}} \lVert\prt_t
\varphi\rVert_{L^\infty(\rline^3)},$$
$$ \lVert\varphi\nabla\chi_\eps\rVert_{L^2(\rline^3)}\leq C
\eps^{\frac{1}{2}}\lVert\varphi\rVert_{L^\infty(\rline^3)},$$
$$ \lVert  \dot h(t) \nabla g_\eps\rVert_{L^2(\rline^3)}\leq C 
\eps^{\frac{1}{2}}\m |\dot h(t)|\m\lVert\varphi\rVert_{L^
\infty(\rline^3)},  $$
\begin{equation*}\label{esti_g_eps}
\lVert \prt_t g_\eps(t, x)\rVert_{L^2(\rline^3)}\leq \eps^{\frac{3}{2}}\lVert\prt_t\tilde g_\eps\rVert_{L^2(\rline^3)},
\end{equation*}
where we used \rfb{ineqvarphi} and Lemma \ref{lemmachi}. 
Combining with all estimates above, we obtain from \rfb{timeprt} that 
\begin{equation*}
\lVert\prt_t\varphi_\eps-\prt_t\varphi\rVert_{
L^2(\rline^3)}\leq\m C\eps^{\frac{3}{2}}\lVert\prt_t\varphi
\rVert_{L^\infty( \rline^3)}
+ C\eps^{\frac{1}{2}}\lVert\varphi\rVert_{L^\infty( \rline^3)}
+C\eps^{\frac{1}{2}}| \dot h(t)|\m\lVert\varphi\rVert_{L^\infty( \rline^3)}+\eps^{\frac{3}{2}}\lVert\prt_t\tilde g_\eps\rVert_{L^2(\rline^3)}.
\end{equation*}
Taking the $L^q$ norm with respect to time $t$ on $(0, T)$ and using the embedding $L^\infty(0, T)\hookrightarrow L^q(0, T)$ for $1<q<\frac{4}{3}$, we have
\begin{equation}\label{middes}
\lVert\prt_t\varphi_\eps-\prt_t\varphi\rVert_{L^q(0, T; 
L^2(\rline^3))}\leq C\eps^{\frac{1}{2}}\lVert\varphi\rVert_{W^{1, \infty}
((0, T)\times \rline^3)}+C\m\eps^{\frac{3}{2}}\lVert\prt_t\tilde g_\eps\rVert_{L^q(0, T;L^2(\rline^3))}.    
\end{equation}
Moreover, note that 
\begin{equation}\label{mid2}
  \lVert\prt_t\tilde \varphi_\eps\rVert_{L^2(\rline^3)}=\lVert 
\nabla\chi\m\prt_t\varphi\rVert_{L^2(\rline^3)}\leq \lVert\prt_t
\varphi\rVert_{L^\infty(\rline^3)}\lVert\nabla\chi
\rVert_{L^2(\rline^3)}\leq C\lVert\prt_t\varphi\rVert_{L^\infty
(\rline^3)}.  
\end{equation} 
Putting \rfb{tildeineq2}, \rfb{middes} and \rfb{mid2} together, we derive that
$$\lVert\prt_t\varphi_\eps-\prt_t\varphi\rVert_{L^q(0, T; 
L^2(\rline^3))}\leq C\eps^{\frac{1}{2}}\lVert\varphi\rVert_{W^{1, \infty}
((0, T)\times \rline^3)}. $$
This immediately gives that
$$ \prt_t\varphi_\eps\longrightarrow \prt_t\varphi\quad 
\text{strongly in}\quad L^q(0, T; L^2(\rline^3)).$$

Finally, according to the structure of $\varphi_\eps$ in 
\rfb{phiepss}, it is obvious to see that \rfb{divfree} 
and \rfb{inner} hold, which ends the proof.
\end{proof}

\begin{rmk}
{\rm Compared with the result in \cite{lacave2017small}, 
here we have strong convergence of the approximate sequence 
$(\varphi_\eps)$ in $L^\infty(0, T; H^1(\rline^3))$ 
and $L^q(0, T; L^2(\rline^3))$. This benefits from that 
we consider the convergence in $L^2(\rline^3)$, 
rather than in $L^3(\rline^3)$, which is 
to match the regularity of the weak solution for the 
fluid-body system and for the Navier-Stokes equations.
}
\end{rmk}

\begin{rmk}\label{remark_varphi}
{\rm We claim that the sequence $(\varphi_\eps)$ constructed 
in Proposition \ref{Prop_test-func} is an admissible test 
functions for fluid-body system. Indeed, we know from the 
strong convergence $h_\eps\longrightarrow h$ that there 
exists $\delta_0>0$, such that for every $\eps<\delta_0$, 
we have $|h_\eps-h|\leq \frac{1}{2}\eps$. In this case, 
we obtain from \rfb{inner} that 
$$ \varphi_\eps\equiv 0\quad \text{in} \quad B(h_\eps(t), 
\eps).$$
}
\end{rmk}


\section{Strong convergence}\label{section-strong}
In order to pass to the limit for the non-linear term, we derive in this part the strong convergence of some sub-sequence of $(u_{\eps})$ in $L^2(0,T;L^2\loc(\rline^3))$. Precisely, we prove the following result. 
\begin{prop}\label{prop_conv}
There exists a sub-sequence $(u_{\eps_k})$ of $(\ue)$ which converges strongly in $L^2(0,T;L^2_{\rm loc} (\rline^3))$. 
\end{prop}
\begin{proof}
We derive a time estimate and use the Arzel\`a-Ascoli lemma to get a strong convergence. By the definition of the test function constructed in Section \ref{section-testfunc} (see \eqref{phiepss}), we find that even if $\varphi$ does not depend on time variable, the modified test function $\varphi_\eps$ still depends on. Let $\varphi \in C^\infty_c (\rline^3)$ be a test function with $\div \varphi =0$, we construct a new test function $\varphi_\eps$ as in Section \ref{section-testfunc},
\begin{equation}\label{phiepss-strong}
\varphi_\eps(t,x):=\varphi(x)\chi_\eps(t,x)
-g_\eps(t,x),
\end{equation}
where the function $\chi_\eps$ and $g_\eps$ has been introduced in \rfb{chi} and \rfb{phiepss}, respectively. It is not difficult to see that $\varphi_\eps (t, x)$ defined in \rfb{phiepss-strong} satisfies the properties of Proposition \ref{Prop_test-func}.

We first estimate the following term:
\begin{align*}
\Bigl| \int_{\rline^{3}}\ue(t,x)\cdot \varphi_\eps(t,x) \,\dd x\Bigr|
&= \Bigl|\int_{\rline^{3}}\ue\cdot \left(\varphi(x)\chi_\eps(t,x)
-g_\eps(t,x) \right)
 \,\dd x\Bigr|\\
&\leq \norm{\ue}_{L^2} \norm{\varphi}_{L^2} \norm{\chi_\eps}_{L^\infty} 
+ \norm{\ue}_{L^2} \norm{g_\eps}_{L^2} \\
& \leq C \left( \norm{\ue}_{L^2} \norm{\varphi}_{L^2} + \eps^{\frac{3}{2}} \norm{\ue}_{L^2}  \norm{\varphi}_{L^\infty} \right),
\end{align*}
where we used Lemma \ref{lemmachi} and estimate \eqref{ineqvarphi}. Then by the Sobolev embedding $H^2 \hookrightarrow L^\infty$ in 3D, we get that 
\begin{align*}
\Bigl| \int_{\rline^{3}}\ue(t,x)\cdot \varphi_\eps(t,x) \,\dd x\Bigr|
\leq C \norm{\ue}_{L^2} \norm{\varphi}_{H^2}
\leq C_1 \norm{\varphi}_{H^2},
\end{align*}
where we used the boundedness of $\ue$ in $L^\infty(0, T; L^2(\rline^3))$.

Based on the above estimate we find that, for fixed $t$ and for any $\varphi \in C^\infty_{c}(\rline^3)$ with divergence free condition, the map
\begin{equation*}
\varphi \longmapsto \int_{\rline^{3}}\ue(t,x)\cdot\varphi_\eps(t,x)\,\dd x\in\rline^3
\end{equation*}
is linear and continuous with respect to the $H^2$-norm. Therefore, according to the Lax-Milgram theorem, there exists some function $U_\eps(t)\in (H_{\sigma}^{2}(\rline^3))'$, such that
\begin{equation*}
\left\langle U_\eps(t), \varphi\right\rangle= \int_{\rline^{3}}\ue(t,x)\cdot\varphi_\eps(t,x)\,\dd x \FORALL \varphi \in H_\sigma^2(\rline^3). 
\end{equation*}
Based on this notation, we note that for $\psi\in H^2(\rline^3)$ with $\mathbb{P}\psi\in H^2_\sigma(\rline^3)$
$$|\left\langle \mathbb{P} U_\eps, \psi \right\rangle|=|\left\langle U_\eps, \mathbb{P}\psi\right\rangle|=\left|\int_{\rline^3}u_\eps\cdot (\mathbb{P}\psi)_\eps\dx \right|\leq C\lVert\mathbb{P}\psi\rVert_{H^2}\leq C\lVert\psi\rVert_{H^2}.$$
{\color{black} This gives us that
\begin{equation}\label{boundpsiep}
\| \mathbb{P} U_\eps(t)\|_{H_\sigma^{-2}}\leq C_1\FORALL t\geq0,
\end{equation}
where $\mathbb{P}$ denotes the usual Leray projector in $\rline^3$, i.e. the $L^2$-orthogonal projection on the subspace of divergence-free vector fields. }

Recalling Remark \ref{remark_varphi}, for small $\eps>0$, the function $\varphi_\eps(t,x)$ constructed in \rfb{phiepss-strong} is an admissible test functions for the fluid-body system. 
Thereby, using the definition of $\langle U_\eps(t),\varphi\rangle$, we obtain from the weak formulation of $\ue$ on the time interval $(s,t)$ that 
\begin{equation}\label{Udif}
\begin{aligned}
& \langle U_\eps(t)-U_\eps(s),\varphi\rangle \\
= & \int_{\rline^{3}}\ue(t,x)\cdot\varphi_\eps(t,x)\,\dd x - \int_{\rline^{3}}\ue(s,x)\cdot\varphi_\eps(s,x)\,\dd x \\
= & \int_s^t\int_{\rline^{3}}\ue\cdot \partial_\tau\varphi_\eps +\int_s^t\int_{\rline^{3}}\left(\ue\cdot \nabla\right)\ue\cdot\varphi_\eps-2 \nu\int_s^t\int_{\rline^{3}}D(\ue):D(\varphi_\eps). \\
\end{aligned}
\end{equation}
We shall bound the three terms in the right-hand side above. 
Since there is $\nabla \varphi_\eps$ on the right side of \rfb{Udif}, we need $H^1$-estimate of $\varphi_\eps$. 
According to Proposition \ref{Prop_test-func}, we know that $\varphi_\eps$ converges strongly to $\varphi$ in $L^\infty(0,T; H^1(\rline^3))$, which implies that
\begin{equation}\label{varphi_H1}
\norm{\varphi_\eps}_{L^\infty (0, T; H^1 )} \leq \norm{\varphi}_{H^1} 
 + \norm{\varphi_\eps - \varphi}_{L^\infty (0, T; H^1)} \leq C \norm{\varphi}_{H^2}. 
\end{equation}
Thus, by using H\"older's inequality we have 
\begin{equation*}
\begin{aligned}
\left|\nu\int_s^t\int_{\rline^{3}}D(\ue):D(\varphi_\eps)\right|
& \leq \nu \int_s^t \norm{D(\ue)}_{L^2} \norm{D(\varphi_\eps)}_{L^2} \\
& \leq C \nu(t-s)^{\frac{1}{2}} \norm{\varphi}_{H^{2}}  \norm{ \ue}_{L^{2}(0,T; H^1)} \\ 
& \leq C\nu (t-s)^{\frac{1}{2}}\norm{\varphi}_{H^2},
\end{aligned}
\end{equation*}
where we used \rfb{varphi_H1} and the boundedness of $\ue$ in $L^{2}(0,T; H^1(\rline^3))$.
Then we treat the nonlinear term in \eqref{Udif}. Using the Sobolev embedding $H^1\hookrightarrow L^6$ and the Gagliardo-Nirenberg inequality $\norm{f}_{L^{3}} \leq C \norm{f}_{L^{2}}^{\frac{1}{2}} \norm{\nabla f}_{L^{2}}^{\frac{1}{2}}$, we bound 
\begin{equation*}
\begin{aligned}
\left|\int_{s}^{t} \int_{\rline^{3}} \left(\ue \cdot \nabla\right) \ue \cdot \varphi_{\eps}\right| 
&\leq \int_{s}^{t} \norm{\ue}_{L^{3}}  \norm{\nabla \ue}_{L^{2}} \norm{\varphi_{\eps}}_{L^6} \\
&\leq C \int_{s}^{t} \norm{\ue}_{L^{2}}^{\frac{1}{2}} \norm{\nabla \ue}_{L^{2}}^{\frac{3}{2}} \norm{\varphi_{\eps}}_{H^1}.
\end{aligned}
\end{equation*}
Hence, by H\"older's inequality in time, we get that
\begin{align*}
\left|\int_{s}^{t} \int_{\rline^{3}} \left(\ue \cdot \nabla\right) \ue \cdot \varphi_{\eps}\right|  
& \leq C (t-s)^{\frac{1}{4}} \norm{\ue}^{\frac{1}{2}}_{L^{\infty}(0,T;L^2)} \norm{\ue}^{\frac{3}{2}}_{L^{2}(0,T;H^1)} \norm{\varphi_{\eps}}_{L^\infty(0, T; H^1)}\\
& \leq  C (t-s)^{\frac{1}{4}} \norm{\varphi}_{H^2},
\end{align*}
where we used the estimate \eqref{varphi_H1} and the boundedness of $\ue$ in $L^{\infty}(0,T;L^2) \cap L^{2}(0,T;H^1)$. 

Now, we consider the most complicated term, i.e. the term with time-derivative in the right-hand side of \eqref{Udif}. Recalling the construction of $\varphi_\eps$ (see \eqref{timeprt} for the explicit expression of $\prt_\tau\varphi_\eps$), we have 
\begin{equation}\label{timevarin}
\begin{aligned}
\int_s^t\int_{\rline^{3}}\ue\cdot\partial_\tau\varphi_\eps
= & \int_s^t\int_{\rline^{3}}\ue \cdot \left(\chi_\eps\m\prt_\tau\varphi+
\varphi\m\nabla\chi_\eps-\prt_\tau g_\eps(\tau, x)-\nabla 
g_\eps\dot h(\tau)\right).
\end{aligned}
\end{equation}
We conclude from the proof of Proposition \ref{Prop_test-func} that 
$$\lVert\prt_\tau  g_\eps\rVert_{L^2(\rline^3)}
\leq C\eps^{\frac{3}{2}}\lVert\prt_\tau \varphi\rVert_{L^\infty (\rline^3)}. $$
As $\varphi$ does not depend on time, thereby we observe that the first term and the third term on the right side of \rfb{timevarin} disappear. Thus, we need to bound the remaining two terms. Then \rfb{timevarin} becomes 
\begin{equation}\label{relation_time_var}
\begin{aligned}
\int_s^t\int_{\rline^{3}}\ue\cdot\partial_\tau\varphi_\eps
= \int_s^t\int_{\rline^{3}}\ue\cdot
(\varphi\nabla\chi_\eps) - \int_s^t\int_{\rline^{3}}\ue \cdot
\left( \nabla 
g_\eps\dot h(\tau)\right). 
\end{aligned}
\end{equation}
Recalling the property of the cut-off function $\chi_\eps$ in Lemma \ref{lemmachi}, we derive that
\begin{equation*}
\begin{aligned}
\left|\int_s^t\int_{\rline^{3}}\ue\cdot
(\varphi\m\nabla\chi_\eps) \right|
& \leq \int_s^t \norm{\ue}_{L^2} \norm{\varphi}_{L^\infty} \norm{\nabla\chi_\eps}_{L^2} \\
& \leq C \eps^{\frac{1}{2}} (t-s) \norm{\varphi}_{H^2}  \norm{\ue}_{L^\infty(0, T; L^2)}
\\
& \leq \tilde C (t-s) \m\norm{\varphi}_{H^2}.
\end{aligned}
\end{equation*}
For the second term in the right-hand side of \eqref{relation_time_var}, we bound
\begin{equation*}
\begin{aligned}
\left|\int_s^t\int_{\rline^{3}}\ue \cdot 
\left(\nabla  g_\eps(\tau,y)\dot h(\tau)\right) \right|
& \leq \int_s^t \norm{\ue}_{L^2} \norm{\dot h(\tau) \nabla g_\eps}_{L^2} \\
& \leq C \eps^{\frac{1}{2}} \lVert\varphi\rVert_{H^2} \int_s^t \norm{\ue}_{L^2}  \m |\dot h(\tau)|
\\
& \leq C \eps^{\frac{1}{2}}  (t-s)^{1/q'} \norm{\varphi}_{H^2} \norm{h}_{W^{1,q}} \norm{\ue}_{L^\infty(0, T; L^2)} \\
& \leq \tilde C (t-s)^{1/q'} \norm{\varphi}_{H^2} \quad\quad 4<q'<\infty,
\end{aligned}
\end{equation*}
where we used the estimates of $g_\eps$ in  \rfb{ineqvarphi}. 
Gathering the two estimates above, we get that 
\begin{equation*}
\left|\int_s^t\int_{\rline^{3}}\ue\cdot\partial_\tau\varphi_\eps \right|
\leq \tilde C (t-s)^{1/q'} \norm{\varphi}_{H^2},
\end{equation*}
with $4<q'<\infty$.
Consequently, combining all the estimates above, we have, for some constant $C$ depending on $\nu$,
\begin{equation*}
\begin{aligned}
\left|\langle U_\eps(t)-U_\eps(s),\varphi\rangle \right|
& \leq  \left( C \nu (t-s)^{1/2} + C (t-s)^{1/4}+ \tilde C (t-s)^{1/q'} \right) \norm{\varphi}_{H^2}\\
& \leq C (t-s)^{1/q'} \norm{\varphi}_{H^2}.
\end{aligned}
\end{equation*}
It implies that $(\mathbb{P} U_\eps)$ is equi-continuous in time with value in $H_\sigma^{-2}$. Since $(\mathbb{P}  U_\eps)$ is uniformly bounded in $H^{-2}_{\sigma}$ (see \eqref{boundpsiep}), according to the Arzel\`a-Ascoli theorem and compact embedding $H^{-2} \hookrightarrow H^{-3}\loc $, we are able to extract a sub-sequence $( \mathbb{P} U_{\eps_k})$ of $(\mathbb{P}  U_\eps)$ which converges strongly in $ C^0(0,T;H^{-3}\loc)$. We denote the limit of $(\mathbb{P} U_{\eps_k})$ by $ u$, i.e.
\begin{equation}\label{limee}
\mathbb{P} U_{\eps_k}\longrightarrow  u \quad \text{strongly in}\quad  L^\infty(0, T; H^{-3}\loc(\rline^3)).
\end{equation}

{\color{black}
As $\varphi$ is divergence free, we have $\mathbb{P} \varphi = \varphi$. It then follows from the property of the Leray projector that
$$
\langle \mathbb{P} U_{\eps_k} - u_{\eps_k} , \varphi \rangle = \langle U_{\eps_k} - u_{\eps_k} , \varphi \rangle.
$$
Recalling the definition of $U_{\eps_k}$ and the test function $\varphi_{\eps_k}$, we can bound
\begin{equation*}
\begin{aligned}
\left|\langle \mathbb{P} U_{\eps_k} - u_{\eps_k} , \varphi \rangle \right|
& = \left|\int_{\rline^{3}}u_{\eps_k}(t,x)\cdot \left(\varphi_{\eps_k}(t,x) - \varphi (x)\right)\,\dd x \right| \\
& \leq \norm{u_{\eps_k}}_{L^2} \norm{\varphi_{\eps_k} - \varphi}_{L^2} \\
& \leq \norm{u_{\eps_k}}_{L^2} \norm{\varphi \left(\chi_{\eps_k} -1 \right)
-g_{\eps_k}}_{L^2} \\
& \leq \norm{u_{\eps_k}}_{L^2} \norm{\varphi}_{L^\infty}\norm{\chi_{\eps_k} -1}_{L^2}
+ \norm{u_{\eps_k}}_{L^2} \norm{g_{\eps_k}}_{L^2}  \\
& \leq C {\eps_k}^{3/2} \norm{u_{\eps_k}}_{L^2} \norm{\varphi}_{L^\infty} \\
& \leq C {\eps_k}^{3/2} \norm{u_{\eps_k}}_{L^2} \norm{\varphi}_{H^2},
\end{aligned}
\end{equation*}
where we used Lemma \ref{lemmachi} and \rfb{ineqvarphi}. By the density of $C^\infty_{0, \sigma}$ in $H^2_\sigma$, it gives that 
$$
\norm{\mathbb{P} U_{\eps_k} - u_{\eps_k}}_{H^{-2}} \leq C {\eps_k}^{3/2} \norm{u_{\eps_k}}_{L^2}.
$$
As $u_{\eps_k}$ is bounded in $L^\infty(0, T; L^2)$, we get that 
$$
\norm{\mathbb{P} U_{\eps_k} - u_{\eps_k}}_{L^\infty(0, T; H^{-2})} \leq C {\eps_k}^{3/2} \norm{u_{\eps_k}}_{L^\infty(0, T; L^2)} \toep 0.
$$
This, together with \rfb{limee}, implies that
$$
u_{\eps_k} \toepk  u \quad \text{in} \quad L^\infty(0, T; H^{-3}\loc(\rline^3)).
$$
Finally, by using the interpolation inequality:
$$\lVert u_{\eps_k}\rVert_{L^{\frac{8}{3}}(0, T; L_{\rm loc}^2(\rline^3))}\leq C \lVert u_{\eps_k}\rVert_{L^\infty(0, T; H_{\rm loc}^{-3}(\rline^3))}^{\frac{1}{4}}\lVert u_{\eps_k}\rVert_{L^2(0, T; H_{\rm loc}^1(\rline^3))}^{\frac{3}{4}}, $$
we conclude that 
\begin{equation*}
u_{\eps_k} \longrightarrow  u\quad\text{strongly in }  L^{2}(0,T;L^2\loc(\rline^3)).
\end{equation*}
This completes the proof of Proposition \ref{prop_conv}.}
\end{proof}

\section{Proof of the main Theorem}\label{section-limit}
Combined with the modified test function constructed in Proposition \ref{Prop_test-func} and the strong convergence in Proposition \ref{prop_conv}, now we are able to prove our main result, i.e. Theorem \ref{main-thm}. 
\begin{proof}[Proof of Theorem \ref{main-thm}]
We aim to prove that the limit $u$ obtained in Proposition \ref{prop_conv} is a solution of the Navier-Stokes equations in $\rline^3$ with initial data $u_0(x)$. For every fixed and finite $T > 0$, we need to show that the limit $u$ verifies 
\begin{equation*}
-\int_{0}^{T} \int_{\rline^3} u \cdot \left(\prt_s
\m\varphi  + (u \cdot \nabla) \m\varphi
\right) \, \dx \m\ds
+ \nu \int_{0}^{T} \int_{\rline^{3}} \nabla u : \nabla \varphi  
\,  \dx \m\ds 
= \int_{\rline^{3}}\m u^0(x) \cdot \varphi(0, x) 
\,  \dx,
\end{equation*}
for every $\varphi\in C_c^1([0,T); H_\sigma^1(\rline^3)).$

Based on our assumptions in Theorem \ref{main-thm}, we know that 
\begin{align*}
\ue \; \text{is bounded in} \; L^{\infty}\left(0,T; L^{2}(\rline^3)\right) \cap L^{2}\left(0,T; H^{1}(\rline^3)\right).
\end{align*}
According to weak compactness, for the limit $u$ in Proposition \ref{prop_conv}, there exists sub-sequence $(u_{\eps_k})$ of $(\ue)$ such that
\begin{gather}
u_{\eps_k} \rightharpoonup u \quad\text{weak-$\ast$ in } L^\infty(0,T;L^2(\rline^3))\notag,\\
u_{\eps_k} \rightharpoonup u \quad\text{weakly in } L^2(0,T;H^1(\rline^3)) \label{weakcon}.
\end{gather}
The above subsequence $(u_{\eps_k})$ can be chosen by diagonalization method. 

Let $\varphi \in C^\infty_c ([0, T)\times \rline^3)$ with $\div \varphi = 0$. For small $\eps>0$, we can take the family $(\varphi_\eps)_{\eps>0}$ obtained in Proposition \ref{Prop_test-func} as the test functions of the fluid-solid system (see Remark \ref{remark_varphi}). According to Definition \ref{defu}, we have
\begin{multline}\label{ineq_conve}
-\int_{0}^{T} \int_{\rline^3} \ue \cdot \prt_s
\m\varphi_{\eps} \, \dx \m\ds -\int_{0}^{T} \int_{\rline^3} \ue \cdot (\ue \cdot \nabla) \m\varphi_{\eps} \, \dx \m\ds \\
+ 2 \nu \int_{0}^{T} \int_{\rline^{3}} D (\ue) : D (\varphi_{\eps})  
\,  \dx \m\ds 
= \int_{\rline^{3}} \m \ue^0(x) \cdot \varphi_{\eps}(0, x) 
\,  \dx. 
\end{multline}
We pass to the limit $\eps \to 0$ for the four terms above.\\
Firstly, recalling the convergence results \rfb{timeconv} in Proposition \ref{Prop_test-func} and the continuous embedding $W^{1, q}[0, T)\hookrightarrow C^0[0, T)$, we have 
\begin{align*}
\varphi_{\eps} (0, x) \longrightarrow \varphi(0) \;\text{strongly in}\; L^{2}(\rline^{3}).
\end{align*}
Together with the assumption that $\ue^0(x)$ converges weakly to $u^0(x)$ in $L^2(\rline^3)$, we derive that
\begin{align}{\label{converge1}}
\int_{\rline^{3}} u_{\eps}^0(x) \cdot \varphi_{\eps}(0, x) \,  \dx\longrightarrow \int_{\rline^{3}} u^0(x) \cdot \varphi(0, x) \,  \dx.
\end{align}
Putting together the weak convergence in \eqref{weakcon} and the strong convergence of $\varphi_{\eps}$ in \rfb{convergphi},
we immediately obtain that
\begin{align*}
\int_{0}^{T}  \int_{\rline^{3}} D (u_{\eps}) :  D(\varphi_{\eps}) \m \dx \m\ds  \longrightarrow \int_{0}^{T}  \int_{\rline^{3}} D(u) : D( \varphi) \m \dx \m\ds .
\end{align*}
Moreover, we have 
$$
 \int_{\rline^{3}} D(u) : D( \varphi) \m \dx = \frac{1}{2} \int_{\rline^{3}} \nabla u :  \nabla \varphi \m \dx + \frac{1}{2} \int_{\rline^{3}} \div u \, \div \varphi \m \dx =  \frac{1}{2} \int_{\rline^{3}} \nabla u : \nabla \varphi \m \dx.
$$
Thus, we further have
\begin{align}{\label{converge2}}
2 \nu \int_{0}^{T}  \int_{\rline^{3}} D (u_{\eps}) :  D(\varphi_{\eps}) \m \dx \m\ds  \longrightarrow \nu \int_{0}^{T}  \int_{\rline^{3}} \nabla u : \nabla \varphi \m \dx \m\ds .
\end{align}
For the term with time derivative in \eqref{ineq_conve},
we decompose it as follows:
\begin{equation*}
\begin{aligned}
\int_{0}^{T} \int_{\rline^3} u_{\eps_k} \cdot 
\prt_s\m \varphi_{{\eps}_k} \, \dx \m\ds 
=& \int_{0}^{T} \int_{\rline^3} u_{{\eps}_k} \cdot (\prt_s\m \varphi_{{\eps}_k} - \prt_s\m \varphi)\, \dx \m\ds \\
&+ \int_{0}^{T} \int_{\rline^3} (u_{\eps_k} - u ) \cdot \prt_s\m \varphi \, \dx \m\ds 
+ \int_{0}^{T} \int_{\rline^3} u \cdot \prt_s\m \varphi \, \dx \m\ds\\
\end{aligned}
\end{equation*}
The strong convergence of $\prt_t \varphi_{{\eps}_k}$ in \rfb{timeconv}, together with the boundedness of the subsequence $u_{\eps_k}$ in $L^\infty (0, T; L^2)$, implies that
\begin{equation*}
\begin{aligned}
\int_{0}^{T} \int_{\rline^3} u_{{\eps}_k} \cdot (\prt_s\m \varphi_{{\eps}_k} - \prt_s\m \varphi)\, \dx \m\ds 
& \leq  \int_{0}^{T} \norm{u_{\eps_k}}_{L^2} \norm{\prt_s\m \varphi_{{\eps}_k} - \prt_s\m \varphi}_{L^2} \m \ds \\
& \leq  T^{\frac{1}{q'}} \norm{u_{\eps_k}}_{L^\infty (0, T; L^2)} \norm{\prt_s\m \varphi_{{\eps}_k} - \prt_s\m \varphi}_{L^q (0, T; L^2)}\\
& \toepk 0,
\end{aligned}
\end{equation*}
where $q'$ is the conjugate of $q$. \\
Note that $\varphi \in C^\infty_c ([0, T)\times \rline^3)$, according to the strong convergence of $u_{\eps_k}$ in $L^2(0,T;L^2\loc)$ in Proposition \ref{prop_conv}, we obtain that
\begin{equation*}
\begin{aligned}
\int_{0}^{T} \int_{\rline^3} (u_{\eps_k} - u ) \cdot \prt_s\m \varphi \, \dx \m\ds 
& \leq \int_{0}^{T} \norm{u_{\eps_k} - u}_{L^2\loc} \norm{\prt_s\m \varphi}_{L^2} \m \ds \\
& \leq  \norm{u_{\eps_k} - u}_{L^2 (0, T; L^2\loc)} \norm{\prt_s\m \varphi}_{L^2 (0, T; L^2)}\\
& \toepk 0.
\end{aligned}
\end{equation*}
Gathering the estimates above, we get that 
\begin{equation}\label{convergence3}
\begin{aligned}
\int_{0}^{T} \int_{\rline^3} u_{\eps_k} \cdot 
\prt_s\m \varphi_{{\eps}_k} \, \dx \m\ds 
\longrightarrow \int_{0}^{T} \int_{\rline^3} u \cdot \prt_s\m \varphi \, \dx\m\ds.\\
\end{aligned}
\end{equation}

Now it remains to treat the nonlinear term in \eqref{ineq_conve}. That is to prove the convergence
\begin{equation}\label{convergence4}
\begin{aligned}
\int_{0}^{T} \int_{\rline^3} (u_{\eps_k} \otimes u_{\eps_k}) :
\nabla \varphi_{{\eps_k}} \,\dx \m\ds 
\toepk \int_{0}^{T} \int_{\rline^3} (u \otimes u) : \nabla \varphi \, \dx\m\ds.\\
\end{aligned}
\end{equation}
To do this, we decompose it as 
\begin{equation*}
\begin{aligned}
\int_{0}^{T} \int_{\rline^3} (u_{\eps_k} \otimes u_{\eps_k}) :
\nabla \varphi_{{\eps_k}} \,\dx \m\ds 
=& \int_{0}^{T} \int_{\rline^3} (u_{\eps_k} \otimes u_{\eps_k}) : (\nabla\varphi_{{\eps_k}} - \nabla \varphi)\, \dx \m\ds\\
& + \int_{0}^{T} \int_{\rline^3} (u_{\eps_k} \otimes u_{\eps_k} - u\otimes u) : \nabla\varphi \, \dx \m\ds \\
& + \int_{0}^{T} \int_{\rline^3} (u \otimes u) : \nabla \varphi \, \dx\m\ds.
\end{aligned}
\end{equation*}
To treat the first term on the right-hand side, we use H\"older's inequality and the Gagliardo–Nirenberg interpolation inequality $\norm{f}_{L^4 (\rline^3)} \leq C \norm{\nabla f}_{L^2 (\rline^3)}^{\frac{3}{4}} \norm{f}_{L^2 (\rline^3)}^{\frac{1}{4}}$ to derive that 
\begin{equation*}
\begin{aligned}
& \int_{0}^{T} \int_{\rline^3} (u_{\eps_k} \otimes u_{\eps_k}) : (\nabla\varphi_{{\eps_k}} - \nabla \varphi)\, \dx \m\ds \\
\leq & \int_{0}^{T} \norm{u_{\eps_k}}_{L^4}^2 \norm{\nabla\varphi_{{\eps_k}} - \nabla \varphi}_{L^2}\\
\leq & \int_{0}^{T} \norm{\nabla u_{\eps_k}}_{L^2}^{\frac{3}{2}} \norm{ u_{\eps_k}}_{L^2}^{\frac{1}{2}} \norm{\nabla\varphi_{{\eps_k}} - \nabla \varphi}_{L^2}\\
\leq & \; T^{\frac{1}{4}}\norm{u_{\eps_k}}_{L^2 (0, T; H^1)}^{\frac{3}{2}} \norm{ u_{\eps_k}}_{L^\infty(0, T; L^2)}^{\frac{1}{2}} \norm{\varphi_{{\eps_k}} - \varphi}_{L^\infty(0,T; H^1)} 
\toepk\m 0,
\end{aligned}
\end{equation*}
where we used the strong convergence of $\varphi_{{\eps_k}}$ in $L^\infty(0,T; H^1)$ (see \eqref{convergphi} in Proposition \ref{Prop_test-func}) and the boundedness of $u_{\eps_k}$ in $ L^{\infty}\left(0,T; L^{2}\right) \cap L^{2}\left(0,T; H^{1}\right)$. \\
For the second term on the right-hand side, we write it as 
\begin{equation}\label{prolast}
\begin{aligned}
&\int_{0}^{T} \int_{\rline^3} (u_{\eps_k} \otimes u_{\eps_k} - u\otimes u) : \nabla\varphi \, \dx \m\ds \\
&=  \int_{0}^{T} \int_{\rline^3} (u_{\eps_k} - u )\otimes u_{\eps_k}  : \nabla\varphi \, \dx \m\ds 
 + \int_{0}^{T} \int_{\rline^3} u \otimes (u_{\eps_k} - u) : \nabla\varphi \, \dx \m\ds.
 \end{aligned}
\end{equation}
Recalling that $\varphi$ is compactly supported, by H\"older's inequality and the strong convergence of $u_{\eps_k}$ in $L^2(0,T;L^2\loc)$, we have
\begin{equation*}
\begin{aligned}
& \int_{0}^{T} \int_{\rline^3} (u_{\eps_k} - u)\otimes u_{\eps_k} : \nabla\varphi \, \dx \m\ds \\
\leq & \int_{0}^{T} \norm{u_{\eps_k} - u}_{L^2\loc} \norm{u_{\eps_k}}_{L^2}  \norm{\nabla\varphi}_{L^\infty} \\
\leq & \; T^{\frac{1}{2}} \norm{u_{\eps_k} - u}_{L^2 (0, T; L^2\loc)} \norm{u_{\eps_k}}_{L^\infty(0, T; L^2)}  \norm{\nabla\varphi}_{L^\infty ((0, T) \times \rline^3)} \toepk 0.
\end{aligned}
\end{equation*}
Doing a similar manner for the last term in \rfb{prolast}, the convergence \eqref{convergence4} holds by combining the above estimates. Gathering the convergence \eqref{converge1}-\eqref{convergence4}, we conclude that the limit $u$ verifies \eqref{energy-eq-u} and it is a Leray weak solution of the Navier-Stokes equations.

This ends the proof of Theorem \ref{main-thm}.
\end{proof}


\section*{Statements and declarations}
The authors declare that they have no known competing financial interests or personal relationships that could have appeared to influence the work reported in this paper. \\

Availability of data and material : Not applicable.


\end{document}